\documentclass[a4paper,12pt]{article}
\usepackage[top=2.5cm,bottom=2.5cm,left=2.5cm,right=2.5cm]{geometry}

\usepackage{authblk}

\usepackage{amssymb,mathrsfs}
\usepackage{graphicx}
\usepackage{amsmath,amsthm,amssymb,lineno}
\usepackage{latexsym}
\usepackage{graphicx,booktabs,multirow}

\usepackage{lipsum}
\usepackage{enumitem}
\setlist[enumerate]{itemsep=0mm}

\usepackage{txfonts,pxfonts,upgreek} 
\usepackage{float}

\usepackage{tikz}
\usetikzlibrary{intersections, spath3}
\usetikzlibrary{calc}
\usetikzlibrary{shapes.geometric,arrows.meta}
\usetikzlibrary{decorations.pathmorphing}
\usetikzlibrary{decorations.markings}
\usetikzlibrary{hobby}
\usepgflibrary{bbox}
\usepackage{graphicx,booktabs,multirow}
\usepackage{appendix}
\usepackage{yhmath}
\usepackage{adjustbox}

\usepackage{color}

\usepackage{setspace}

\newtheorem{thm}{Theorem}[section] %
\newtheorem{theorem}[thm]{Theorem}
\newtheorem{lemma}[thm]{Lemma}

\newtheorem{prob}[thm]{Problem}
\newtheorem{corollary}[thm]{Corollary}
\newtheorem{claim}{Claim}
\newtheorem{proposition}[thm]{Proposition}

\usetikzlibrary{backgrounds}
\usetikzlibrary{arrows}
\usetikzlibrary{shapes,shapes.geometric,shapes.misc}
\pgfdeclarelayer{edgelayer}
\pgfdeclarelayer{nodelayer}
\pgfsetlayers{background,edgelayer,nodelayer,main}
\tikzstyle{none}=[inner sep=0mm]

\usepackage[colorlinks,
linkcolor=red!60!black,
anchorcolor=blue,
citecolor=blue!60!black
]{hyperref}
\definecolor{lightpurple}{rgb}{0.8, 0.7, 1.0} 
\definecolor{skyblue}{rgb}{0.529, 0.807, 0.922}

\newcommand\red[1] {{\color{red} #1}}

\newcounter{countcase}
\def\incase{\addtocounter{countcase}{1}{\noindent {\bf Case \thecountcase}: }}

\newcounter{countclaim}
\def\inclaim{\addtocounter{countclaim}{1}
	{\vspace{0.2 cm}\noindent {\bf Claim \thecountclaim}: }}

\newcommand{\proofend}{{\hfill$\Box$}}

\def \border {\mbox{bd}}
\def \nest {\upomega}
\def \od {\mbox{Ord}}

\def \Nest{\mathbb{W}}
\def \setn{{\cal N}}
\def \setr{{\cal R}}

\newcommand\equ[2] 
{
\begin{equation}
	\label{#1}
	#2
\end{equation}
}

\newcommand\Ceil[1] 
{
	\left \lceil #1
\right \rceil 
}

\newcommand\eqn[2] 
{
	\begin{eqnarray}
		\label{#1}
		#2
	\end{eqnarray}
}

\newcommand \brk[1]
{
	\llbracket #1\rrbracket
}

\let\oldbibliography\thebibliography
\renewcommand{\thebibliography}[1]{%
  \oldbibliography{#1}%
  \setlength{\itemsep}{-2pt}%
  \setlength{\baselineskip}{11pt}
  \setlength{\lineskiplimit}{-\maxdimen}
}

\begin{document}

\title{\bf
Determining the minimum size 
of 
maximal 1-plane graphs\thanks{The work was supported by the National
Natural Science Foundation of China (Grant No. 12271157, 12371346, 12371340) and the Postdoctoral Science Foundation of China  (Grant  No. 2024M760867).}
}

\author[1]{\small Yuanqiu Huang\thanks{E-mail: hyqq@hunnu.edu.cn}}

\author[2]{\small  Zhangdong Ouyang\thanks{Corresponding author.
	Email:oymath@163.com }}

\author[3]{\small 
	Licheng Zhang\thanks{E-mail: lczhangmath@163.com}}

\author[4]{\small Fengming Dong\thanks{Email: fengming.dong@nie.edu.sg  and  donggraph@163.com.}}

\affil[1] {\footnotesize Department of Mathematics, Hunan Normal University, Changsha 410081, P.R.China}

\affil[2] {\footnotesize Department of Mathematics, Hunan First Normal University , Changsha 410205, P.R.China}

\affil[3]{\footnotesize School of Mathematics, Hunan University, Changsha 410082, P.R.China}
	
\affil[4]{\footnotesize National Institute Education, Nanyang Technological University, Singapore}

\date{}
\maketitle

\begin{abstract}
A {\it $1$-plane graph} is a graph together with a drawing 
in the plane in such a way that each edge is crossed at most once. 
A $1$-plane graph is maximal if no edge can be added without violating either 1-planarity or simplicity.
Let $m(n)$ denote
the minimum size of 
a  maximal $1$-plane graph of order $n$.
 Brandenburg et al. established  that  
 $m(n)\ge 2.1n-\frac{10}{3}$
 for all $n\ge 4$, 
 which was improved by 
 Bar\'{a}t and T\'{o}th to 
 $m(n)\ge \frac{20}{9}n-\frac{10}{3}$. 
 In this paper, we confirm that 
 $m(n)=\Ceil{\frac{7}{3}n}-3$
 for all $n\ge 5$.
\end{abstract}


\makeatletter
\renewcommand\@makefnmark%
{\mbox{\textsuperscript{\normalfont\@thefnmark)}}}
\makeatother


\section{Introduction}\label{sec1}

\subsection{Lower bounds of the sizes of  maximal  1-plane (or 1-planar) graphs}

All graphs considered here are simple and finite. 
Let $G=(V,E)$ be a graph. 
Its order and size are $|V|$ and $|E|$,
respectively. 
A {\it drawing} of $G$ 
is a mapping $D$ that assigns to each vertex in $V$ a
distinct point in the plane and to each edge $uv$ in $E$ a
continuous arc connecting $D(u)$ and $D(v)$. We often make no
distinction between a graph-theoretical object (such as a vertex or
an edge) and its drawing. All drawings considered here ensure that no edge crosses itself, no
two edges cross more than once, and no two edges incident with the
same vertex cross each other. A graph is {\it planar} if it can be drawn in the plane without edge crossings.  
A drawing of a graph is $k$-{\it planar} if each of its edges is crossed at most $k$ times. A graph is $k$-{\it planar} if it has a $k$-planar drawing, and it is {\it  maximal $k$-planar} if we cannot add any edge to it so that the resulting graph is still $k$-planar. A graph together with a $k$-planar drawing is a $k$-plane graph, and it is {\it maximal $k$-plane} if we cannot add any edge to it so that the resulting drawing is still $k$-plane. Specifically for $k = 1$,
the notion of 1-planarity was introduced in 1965 by
Ringel \cite{GR}.
Since then many properties of 1-planar graphs have
been studied (e.g. see the survey paper \cite{SK} or the book \cite{SH}).

Determining the density for a finite family of graphs is one of the most fundamental problems in graph theory.  For graphs with fixed  drawings, particularly for $k$-planar graphs, the  extremal  problem for edges  has attracted much interest, see \cite{AAB,CBA, MAB,CP,EH,MHO, YZD, Ka}.
It is known that every maximal planar graph with $n$  $(\ge 3)$  vertices has the same number of
edges, $3n-6$. However, there is no analogous conclusion for 1-planar graphs.
It is well-known that 
any maximal 1-planar graph with $n$ $(\ge 3)$ vertices has at most $4n-8$ edges which is
tight for $n=8$ and $n\ge 10$ (see \cite{HR,IT, JP}). 
But Brandenburg et al. \cite{Br} found
 arbitrarily large sparse maximal 1-planar graphs of size much less than $4n-8$.
 
 \begin{theorem}[\cite{Br}]\label{thm-1}
 \begin{enumerate}
 	\item[(a)] For any integer $n$ 
 	with $n\ge 27$ and $n\equiv 3\pmod{6}$, there exists a 
maximal $1$-plane graph
of order $n$ and size 
 $\frac 73 n-3$.
 \item[(b)] 
 There are arbitrarily large 
 maximal $1$-planar graphs 
 of order $n$ and size 
 $\frac {45}{17}n-\frac{84}{17}$.
 	 \end{enumerate}
 \end{theorem}

Introduced in \cite{EH},  for any 
integer $n\ge 1$,
$m(n)$ denotes 
the minimum value of $|E(G)|$
over all maximal $1$-plane graphs
$G$ of order $n$.
For any integers $a$ and $b$, let $\brk{a,b}$ denote the set 
of integers $i$ with $a\le i\le b$,
and $\brk{1,b}$ is simply written as 
$\brk{b}$.
Clearly, $m(n)={n\choose 2}$ for 
each $n\in \brk{4}$.
Brandenburg et al. \cite{Br} 
obtained the following 
lower bounds for $m(n)$.

\begin{theorem}[\cite{Br}]\label{thm0}
For any $n\ge 4$, 	
$m(n)\ge 2.1n-\frac{10}{3}$ .
\end{theorem}

Bar\'{a}t and T\'{o}th \cite{BT} 
noticed that the bound in Theorem \ref{thm0} is 
far from optimal 
and obtained the following conclusion. 

\begin{theorem}[\cite{BT}]\label{thm1}
 For any $n\ge 4$, 
 $m(n)\ge \frac{20}{9} n-\frac{10}{3}$.
 \end{theorem}

In this paper, we introduce the new concept of ``nests"  and apply a more refined $K_4$-extension operation to
determine $m(n)$, as stated below.

\begin{theorem}\label{main1.3}
$m(n)=\Ceil{\frac{7}{3}n}-3$
for all integers $n$ with $n\ge 5$.
 \end{theorem}

\subsection{Terminology and definitions}

In the following, we explain some terminology and definitions. For any additional terminology or definitions not covered here, we refer to \cite{JAB}.

For any graph $G$,  
let $V(G)$ and $E(G)$ be its vertex set and edge set respectively. 
For convenience, 
let $n(G)=|V(G)|$ and $e(G)=|E(G)|$. 
For any vertex $v\in V(G)$, 
let 
$N_G(v)=\{u\in V (G): uv\in E(G)\}$ and 
$N_{G}[v]=N_{G}(v)\cup\{v\}$.
A vertex $v$ in $G$ is called a {\it dominating vertex} if $N_G[v]=V(G)$. 
 For every edge $e\in E(G)$, $G-e$ denotes the graph obtained from $G$ by deleting $e$. 
For a non-empty subset 
$A\subseteq V(G)$,  
let $G[A]$ denote the subgraph of $G$ 
induced by $A$. 

For any two subsets $V_1$ and $V_2$ 
of $V(G)$, let $E_G(V_1, V_2)$ 
 (or simply $E(V_1, V_2)$) 
denote the set of edges in $G$ 
each of which has one end in $V_1$ 
and the other end in $V_2$. 
For any $S\subseteq V(G)$, 
let $\partial_G(S)$ (or simply $\partial(S)$) denote the set 
$E(S, V(G)\setminus S)$. 
For any $v\in V(G)$, 
$\partial_G(\{v\})$ is also written as 
$\partial_G(v)$.
For any non-empty subset $A$ of $E(G)$, 
let $G[A]$ denote the subgraph 
of $G$ induced by $A$, i.e., 
the subgraph with  vertex set 
$\{v\in V(G): \partial(v)\cap A\ne \emptyset\}$ and edge set $A$.


Let $G$ be a 1-plane graph.  An edge $e$ of $G$ is called  {\it clean}  if it
does not cross other edges in $G$; otherwise, it is {\it crossing}. The drawing $G$ partitions the plane $\mathbb{R}^{2}$ into many parts, and each part of $\mathbb{R}^{2}\setminus G$ is  called a  {\it face} of $G$.  
The boundary of a face ${\cal R}$
is denoted by $\border({\cal R})$,
and ${\cal R}$ is called
	a {\it true face} if 
	$\border({\cal R})$ consists of 
	clean edges only,
	and a {\it false face} otherwise.S 
Let $H$ be a subgraph of a 1-plane graph $G$. This means that $H$ is also a 1-plane graph  obtained from $G$ by ignoring all drawn vertices and edges that do not belong to $H$. Assume that edge $e=xy$ crosses edge $e'=st$ at point $\alpha$. 
Then $\alpha x$ and $\alpha y$ are 
called a {\it near half-edge}  and 
a {\it far half-edge}  incident with $x$, respectively,  and both are also called 
the {\it half-edges} of $e$.
Clearly,  $\alpha x$ 
(resp. $\alpha y $) is 
a  far half-edge (resp. near half-edge) incident with $y$.

For each face $R$ of a $1$-plane graph $G$,  the collection of all edge segments, vertices and crossings on  $\border({\cal R})$ can be organized into several  \emph{closed walks} in $G$ traversing along  several  simple closed curves lying just inside the face ${\cal R}$,  and any 
two edge segments of $G$ are said to be \emph{consecutive} on 
 $\border(\setr)$
 if they are consecutive on one 
 of these closed walks.

The remainder of this paper is organized as follows. 
Section~\ref{sec2} presents fundamental results on maximal 
$1$-plane graphs, including the relationship between hermits and nests -- a special type of regions. 
Section~\ref{sec3} introduces $K_4$-extensions (strong, weak, and micro) and describes the rules for generating a $K_4$-extension sequence of subgraphs. 
Section~\ref{sec4} explores key properties of $K_4$-extensions. In Section~\ref{sec5}, we establish an inequality involving the size, order, and nest number of each graph in a $K_4$-extension sequence 
$\{G_i\}_{i\in \brk{0,N}}$. Finally, in Section~\ref{sec6}, we apply the results from Sections~\ref{sec4} and~\ref{sec5} to prove Theorem~\ref{main1.3}.

\section {Preliminaries}\label{sec2}




Throughout this subsection, let $G$ be a maximal 1-plane graph 
with $n(G)\ge 4$.
The following
three lemmas are from \cite{BT,Br} 
and \cite{EH}.

\begin{lemma}[\cite{BT,Br}]\label{adj}  For any face ${\cal R}$ of $G$, 
there are at least two vertices on  $\border({\cal R})$ and 
any two vertices  on  $\border({\cal R})$ are adjacent in $G$.
\end{lemma}

\begin{lemma}[\cite{BT,Br}]\label{K_4} If $ab$ and $cd$ are two edges in $G$ 
	which cross each other, then $G[\{a,b,c,d\}]\cong K_4$. 
\end{lemma}

\begin{lemma}[\cite{EH}] \label{2-con1}
	\begin{enumerate}
		\item[(a)] Each face in $G$ has at most $4$ vertices
		\footnote{Called real vertices in \cite{EH}.}; and 
				\item[(b)] 
		The subgraph $G^*$  induced by all clean edges of $G$ is spanning and $2$-connected.
			\end{enumerate}
\end{lemma}

Applying Lemmas~\ref{K_4} and~\ref{2-con1} (b), 
we can establish 
the following conclusion on the minimum number of clean edges
in an induced subgraph of $G$ incident with a vertex.

\begin{lemma}\label{clean-e}
	Let $G'$ be a vertex-induced subgraph of $G$.  For any $w\in V(G')$,  if $d_{G'}(w)\geq 3$ ,  
	then there are at least two edges in 
	$\partial_{G'}(w)$ which are clean in $G'$. 
\end{lemma}

\begin{proof}  Let $H=G'[N_{G'}[w]]$.
	Clearly,  $w$ is a dominating vertex of $H$ and $n(H)\ge 4$ by the given condition. 
	Let $H'$ be a maximal $1$-plane graph 
	obtained from $H$ by adding 
	some possible new edges. 
	By Lemma~\ref{2-con1} (b), 
	there are at least two edges 
	in $\partial_{H'}(w)$ which are clean in $H$. 
	Since $w$ is dominating in $H$,  
	such two edges are in $\partial_{H}(w)$.
	
	For any $e\in \partial_H(w)$, 
	if $e$ is crossed by edge 
	$v_1v_2$ in $G'$, 
	then, by Lemma~\ref{K_4},
	we have $wv_i\in E(G')$ for $i=1,2$.
	It follows that $v_1,v_2\in V(H)$
	by the definition of $H$,
	implying that  $v_1v_2\in E(H)$.
	Hence any edge incident with $w$ is clean in $H$ if and only if it is clean in $G'$. The lemma then holds.
\end{proof}

It is evident that $\delta(G)\ge 2$. 
A vertex  in $G$ 
of degree $2$ is called a {\it hermit}
(see Figure~\ref{hermit-f}).
The $1$-plane graph obtained 
from $G$ by removing all its hermits
is called the {\it skeleton} of $G$, 
denoted by $\hat{G}$.

\begin{lemma}[\cite{BT,Br}]
	\label{ske-max}
	The skeleton $\hat{G}$ of $G$ is also maximal $1$-plane and each vertex of $\hat{G}$ is of degree at least three.
\end{lemma}

\begin{lemma}[\cite{BT,Br}]\label{hermit1}
	Let $h$ be a hermit of $G$ with 
	$N_G(h)=\{u,v\}$. Then
	$uv\in E(G)$ and both $hu$ and $hv$ are clean edges of $G$, as shown 
	in Figure~\ref{hermit-f}.
\end{lemma}

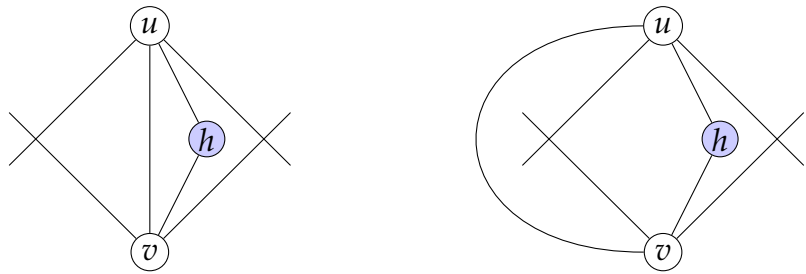
\begin{figure}[H]
	\centering
	\begin{tikzpicture}[scale=0.75]
		
		\tikzset{vertex/.style={circle,draw=black,fill=white,inner sep=2pt},none/.style={fill=white,inner sep=0pt},
			hermit/.style={circle,draw=black,fill=blue!20,inner sep=1pt},
			blackedge/.style={thick}
		}
		\begin{scope}[xshift=-4cm]
			
			
			\fill[white] (-2,0) -- (0,-2) -- (2,0) -- (0,2) -- cycle;
			\node[vertex] (u) at (0,2) {$u$};
			\node[vertex] (v) at (0,-2) {$v$};
			\node[hermit] (h) at (1,0) {$h$};
			
			\node[none] (x1) at (-2.5,0.5) {};
			\node[none] (x2) at (-2.5,-0.5) {};
			\node[none] (y1) at (2.5,0.5) {};
			\node[none] (y2) at (2.5,-0.5) {};
			\draw (u) --(x2);
			\draw (u) -- (y2);
			\draw (v) --(x1);
			\draw (v) --(y1);
			\draw (u) --(v);
			\draw (u) --(h);
			\draw (h) --(v);
			
		\end{scope}

		\begin{scope}[xshift=5cm]

			\fill[white] (-2,0) -- (0,-2) -- (2,0) -- (0,2) -- cycle;
			\node[vertex] (u) at (0,2) {$u$};
			\node[vertex] (v) at (0,-2) {$v$};
			\node[hermit] (h) at (1,0) {$h$};
			
			\node[none] (x1) at (-2.5,0.5) {};
			\node[none] (x2) at (-2.5,-0.5) {};
			\node[none] (y1) at (2.5,0.5) {};
			\node[none] (y2) at (2.5,-0.5) {};
			\draw [ bend left=270, looseness=2.5] (u) to (v);
			\draw (u) --(x2);
			\draw (u) -- (y2);
			\draw (v) --(x1);
			\draw (v) --(y1);
			\draw (u) --(h);
			\draw (h) --(v);
		\end{scope}
	\end{tikzpicture}
	
	\caption{A hermit $h$ surrounded by two pairs of crossing edges}
	\label{hermit-f}
\end{figure}

\begin{lemma}[\cite{BT}]\label{hermit2}
Assume that $h$ is a hermit of $G$ with $N_G(h)=\{u,v\}$.
	Let $G'$ denote the $1$-plane graph 
	obtained from $G$ by deleting 
	$h$ and edge $uv$.
	If ${\cal R}$ is the face of $G'$  containing the point corresponding to $h$, 
	then $u$ and $v$ are the only vertices on $\border({\cal R})$. 
\end{lemma}

Observe that face 
${\cal R}$ described  in Lemma \ref{hermit2} 
is actually bounded by two pairs of half-edges. 
Such a face will be called a nest.

A  {\it nest} of $G$ is a \label{pnest}
region (finite or infinite) of $G$
bounded by four half-edges $\alpha_1u, \alpha_1 v, \alpha_2 u,\alpha_2 v $, 
where $u,v\in V(G)$ and $\alpha_1$
and $\alpha_2$ are crossings,
such that  there are no  vertices and edges of $G$ in the interior of the region,  except possibly for  the edge $uv$,  as shown in Figure~\ref{nest}.
Such a nest is denoted by 
$\setn=\langle \{u, v\}, \{\alpha_1,  \alpha_2\}\rangle$.
Vertices $u$ and $v$ are called 
the {\it supporting  vertices} of $\setn$
and $u\alpha_1, u\alpha_2$, $v\alpha_1$, and $v\alpha_2$  are called the {\it supporting  half-edges}  of $\setn$. 

\begin{figure}[h]
	\usetikzlibrary{backgrounds}
	\centering
	\begin{tikzpicture}[scale=0.75]
		
		\tikzset{vertex/.style={circle,draw=black,fill=white,inner sep=2pt},
			blackedge/.style={}
		}
		\begin{scope}[xshift=-4cm]
			
			
			\fill[gray!20] (-3,-3) rectangle (3,3);
			\fill[white] (-2,0) -- (0,-2) -- (2,0) -- (0,2) -- cycle;
			\node[vertex] (u) at (0,2) {$u$};
			\node[vertex] (v) at (0,-2) {$v$};
			\node at (-2.5, 0) {$\alpha_1$}; 
			\node at (2.5, 0) {$\alpha_2$};  
			
			\node[vertex] (x1) at (-2.5,0.5) {};
			\node[vertex] (x2) at (-2.5,-0.5) {};
			\node[vertex] (y1) at (2.5,0.5) {};
			\node[vertex] (y2) at (2.5,-0.5) {};
			\draw (u) -- node[midway, above left] {$e_1$} (x2);
			\draw (u) -- node[midway, above right] {$e_2$} (y2);
			\draw (v) -- node[midway, below left] {$f_1$} (x1);
			\draw (v) -- node[midway, below right] {$f_2$} (y1);
			
			\node at (0,0) {$\mathcal{N}$};

		\end{scope}
		\begin{scope}[xshift=5cm]
			\draw[dashed] (-3.5,-3) rectangle (3.5,3);
			\coordinate (u) at (0, 2);
			\coordinate (v) at (0, -2);
			\coordinate (x) at (-2, 0.5);
			\coordinate (y) at (-2, -0.5);
			\coordinate (w) at (2, 0.5);
			\coordinate (z) at (2, -0.5);
			
			\draw[blackedge, fill=gray!20] [in=135, out=165, looseness=1.75] (u) to node[pos=0.6, above left] {$e_1$}  (y) [in=-165, out=-135, looseness=1.75] (x) to node[pos=0.6,below left] {$f_1$} (v) [in=45, out=15, looseness=1.75] (u) to node[pos=0.6, above right] {$e_2$} (z) [in=-15, out=-45, looseness=1.75] (w) to node[pos=0.6, below right] {$f_2$}  (v);
			
			\fill[gray!20] (v) -- (z) -- (w) -- (u) -- (x) -- (y) -- (v) ;
			
			\node [vertex] at (u) {$u$};
			\node [vertex] at (v) {$v$};
			\node [vertex] at (x) {};
			\node [vertex] at (y) {};
			\node [vertex] at (w) {};
			\node [vertex] at (z) {};
			\node at (-3,0) {$\mathcal{N}$};
			\node at (-1.9, 0) {$\alpha_1$}; 
			\node at (1.9, 0) {$\alpha_2$};  
		\end{scope}
	\end{tikzpicture}
	
	(a)\hspace{5 cm} (b)
	
	\caption{The  nest $\setn$ 
		(the white region) at (a)  is finite, and  the nest (the white region) $\setn$ at (b) is  infinite, where the possible drawn edge $uv$ is omitted}
	\label{nest}
\end{figure}
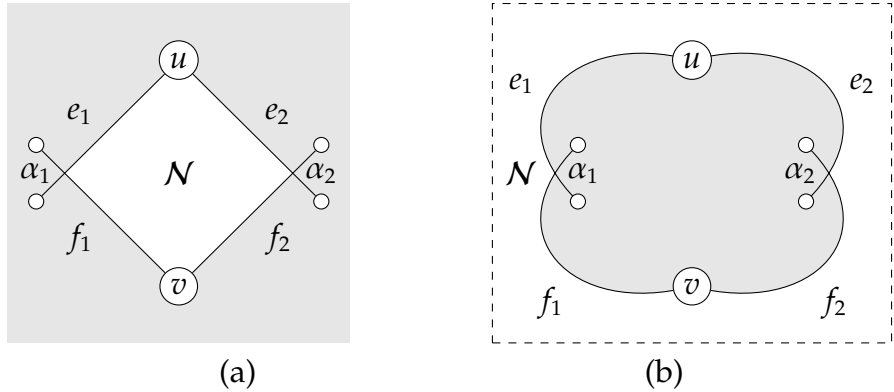

We now end this section by 
applying Lemma~\ref{hermit2}  to get
the following connection between 
hermits and nests. 

\begin{lemma}\label{h-nest}
Each hermit of $G$ is within a nest of 
$\hat G$ and 
there is at most one hermit of $G$ 
within each nest of $\hat G$.
\end{lemma}

\section{$K_4$-extensions and the rules applied}\label{sec3}

Let $G$ be any simple graph, 
$G'$ be a vertex-induced subgraph of $G$ and $F$ be a  $K_4$-subgraph of $G$ with
$1\le |V(F)\cap V(G')|\le 3$.
Let $G''$ denote the graph $G[V(G')\cup V(F)]$.  
We say $G''$ is a 
{\it $K_4$-extension} of $G'$ through $F$, 
and $F$ is called a {\it $K_4$-link} from $G'$. 
There are three types of 
$K_4$-links
 as defined below:
 
 \vspace{-3 mm}
 
\begin{itemize} [itemsep=-1mm]
	\item if $ |V(F)\cap V(G')|=3$, then $F$ is called a {\it  strong $K_4$-link} from $G'$
	(see Figure~\ref{swme} (a)); 
	
	\item if $ |V(F)\cap V(G')|=2$, then $F$ is called a {\it  weak $K_4$-link} from $G'$
	(see Figure~\ref{swme} (b)); 
	and 
	
	\item if $ |V(F)\cap V(G')|=1$, then
	$F$ is called a {\it  micro $K_4$-link} from $G'$
	(see Figure~\ref{swme} (c)).
\end{itemize}

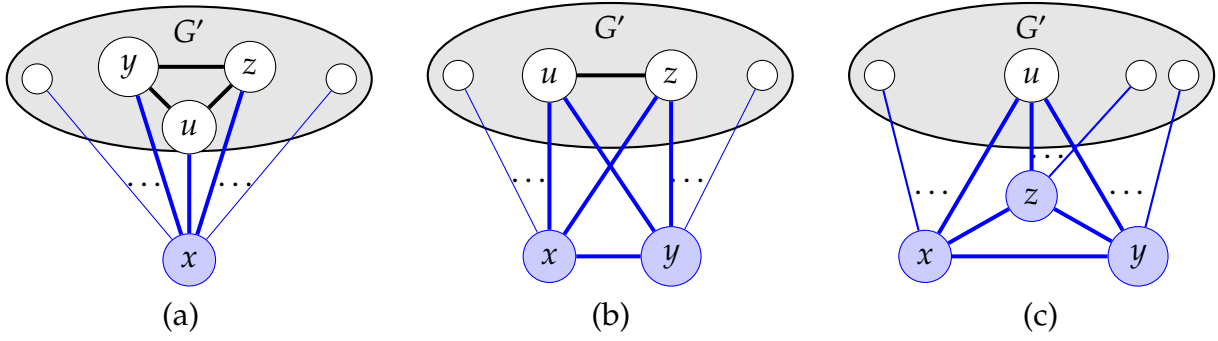
\begin{figure}[!h]
	\centering
	\begin{tikzpicture}[scale=0.8]
		\tikzset{
			whitenode/.style={circle, draw=black, fill=white, minimum size=6mm},
			whitenode2/.style={circle, draw=black, fill=white, minimum size=3mm},
			bluenode/.style={circle, draw=blue, fill=blue!20, minimum size=6mm},
			blueedge/.style={draw=blue},
			blackedge/.style={draw=black, line width=1.5pt},
			blueedge1/.style={draw=blue, line width=1.5pt} %
		}
		\draw[thick,fill=gray!20] (0, 3) ellipse (3.0cm and 1.2cm);
		\node (G) at (0,3.8) {$G'$};
		\node [whitenode] (v) at (-1, 3.2) {$y$}; 
		\node [whitenode] (r) at (1, 3.2) {$z$}; 
		\node [whitenode] (u) at (0, 2.2) {$u$};
		\node [bluenode] (x) at (0, 0) {$x$};
		\node [whitenode2] (left_empty) at (-2.5, 3) {};
		\node [whitenode2] (right_empty) at (2.5, 3) {};
		\node [none] (dots_left) at (-0.7, 1.2) {$\cdots$};
		\node [none] (dots_right) at (0.8, 1.2) {$\cdots$};
		\draw [blackedge] (v) --(u)--(r)--(v)--cycle;
		\foreach \target in {u, v, r} {
			\draw [blueedge1] (x) to (\target);
		}
		\draw [blueedge] (x) to (left_empty);
		\draw [blueedge] (x) to (right_empty);
	\end{tikzpicture}
	\hspace{0.5 cm}
	\begin{tikzpicture}[scale=0.8]
		\tikzset{
			whitenode/.style={circle, draw=black, fill=white, minimum size=5mm},
			whitenode2/.style={circle, draw=black, fill=white, minimum size=3mm},
			bluenode/.style={circle, draw=blue, fill=blue!20, minimum size=5mm},
			blackedge/.style={draw=black, line width=1.5pt},
			blueedge/.style={draw=blue},
			blueedge1/.style={draw=blue, line width=1.5pt} %
		}
		\draw[thick,fill=gray!20] (0, 3) ellipse (3.0cm and 1.2cm);
		
		\node [whitenode] (u) at (-1, 3) {$u$};
		\node [whitenode] (v) at (1, 3) {$z$};
		\node [bluenode] (x) at (-1, 0) {$x$};
		\node [whitenode2] (left_empty) at (-2.5, 3) {};
		\node [whitenode2] (right_empty) at (2.5, 3) {};
		\node  (dots_left) at (-1.3, 1.2) {$\cdots$};
		\node  (dots_right) at (1.3, 1.2) {$\cdots$};
		
		\node [none] (G0_label) at (0, 3.8) {$G'$};
		\node [bluenode] (y) at (1, 0) {$y$};
		\draw [blackedge] (v) to (u);
		\foreach \source/\target in {u/x, x/v, v/y, u/y, x/y} {
			\draw [blueedge1] (\source) to (\target);
		}
		\draw [blueedge] (right_empty) to (y);
		\draw [blueedge] (x) to (left_empty);
	\end{tikzpicture}
\hspace{0.5 cm} 
	\begin{tikzpicture}[scale=0.8]
		\tikzset{
			whitenode/.style={circle, draw=black, fill=white, minimum size=6mm},
			whitenode2/.style={circle, draw=black, fill=white, minimum size=3mm},
			bluenode/.style={circle, draw=blue, fill=blue!20, minimum size=6mm},
			blackedge/.style={draw=black, thick},
			blueedge/.style={draw=blue, thick},
			blackedge1/.style={draw=black, line width=1.5pt},
			blueedge1/.style={draw=blue, line width=1.5pt} %
		}
		\draw[thick,fill=gray!20] (0, 3) ellipse (3.0cm and 1.2cm);
		\node [bluenode] (x) at (-1.75, 0) {$x$};
		\node [bluenode] (y) at (1.75, 0) {$y$};
		\node [bluenode] (z) at (0, 1) {$z$};
		\node [whitenode2] (left_empty) at (-2.5, 3) {};
		\node [whitenode2] (right_empty) at (2.5, 3) {};

		\node  (dots_left) at (-1.6, 1) {$\cdots$};
		\node  (dots_right) at (1.6, 1) {$\cdots$};

		\node [none] (G0_label) at (0, 3.8) {$G'$};
		\node [whitenode] (u) at (0, 3) {$u$};

		\node [whitenode2] (mid_e) at (1.8, 3) {};

		\node (dots_mid) at (0.3,1.6) {$\cdots$};
		
		\draw [blueedge1] (z) to (x);
		\draw [blueedge1] (z) to (y);
		\draw [blueedge1] (u) to (x);
		\draw [blueedge1] (u) to (y);
		\draw [blueedge1] (x) to (y);
		\draw [blueedge1] (u) to (z);
		\draw [blueedge] (z) to (mid_e);
		\draw [blueedge] (left_empty) to (x);
		\draw [blueedge] (right_empty) to (y);
	\end{tikzpicture}
	
\hspace{2 cm}	(a) \hfill (b) \hfill (c)
\hspace{2 cm}{}
	
\caption{$F=G[\{x,y,z,u\}]\cong K_4$
	and $|V(F)\cap V(G')|=3, 2$ or $1$, respectively
}
	\label{swme}
\end{figure}

Accordingly, 
we define strong, weak and 
micro $K_4$-extensions of $G'$
corresponding to 
the types of $K_4$-links of $F$.
For example, if $F$ is a strong $K_4$-link from $G'$, then 
$G''$ is called 
a strong $K_4$-extension of $G'$. 

If $G$ is a connected graph in which 
each edge is contained in 
some $K_4$-subgraph, then for any proper vertex-induced subgraph $G'$, 
there exists a $K_4$-link $F$ from $G'$. 
We will find 
a $K_4$-extension $G''$ of $G'$
via choosing a $K_4$-link $F$ 
determined by the SWM-rule
defined below.

\vspace{0.2 cm}

\noindent
{\bf SWM-rule} (for choosing a $K_4$-link from $G'$):

\vspace{-3 mm}

\begin{itemize}[itemsep=-1mm]
	\item(Rule S)  
	arbitrarily choose a strong $K_4$-link $F$ from $G'$ if such a $K_4$-link exists;

	\item (Rule W) 
	if there are no strong $K_4$-links from $G'$, but weak $K_4$-links from $G'$ exist, we arbitrarily choose a weak $K_4$-link $F$ from $G'$; and 
	
	\item (Rule M) 
	if there are neither strong nor weak $K_4$-links from $G'$, we arbitrarily choose a micro $K_4$-link $F$ from $G'$. 
\end{itemize}

If $G$ is a connected graph 
with $n(G)\ge 4$ and 
each edge in $G$ is contained in 
some $K_4$-subgraph, then it is clear that 
$G$ has a sequence of vertex-induced subgraphs $G_0, G_1, \dots, G_N$, where
$G_0\cong K_4$ and $G_N=G$
such that  for each $i\in \brk{N}$, 
$G_i$ is a $K_4$-extension of 
$G_{i-1}$ through a $K_4$-link $F_i$
which is determined by the SWM-rule. 

We now focus on a maximal $1$-plane graph $G$ with the property that each edge in $G$ is contained 
in some $K_4$-subgraph.\footnote{For any maximal $1$-planar graph $G$, if an edge $e$ in $\hat{G}$ is not contained in any $K_4$-subgraph, then $\hat{G}-e$ is separable 
	(see Figure~\ref{local-str}) 
	and thus  $\hat{G}$ can be decomposed into smaller maximal $1$-planar graphs with the property that each edge is contained in 
	some $K_4$-subgraph.
}
We expect to find such a  sequence of vertex-induced subgraphs $G_0, G_1, \dots, G_N$ of $G$  such that 
for each $i\in \brk{0,N}$,
we can prove 
an inequality,  i.e., the inequality of (\ref{e5-1}) in Section~\ref{sec5}, 
giving a lower bound for $e(G_i)$.
For this purpose, 
we need to strengthen the SWM-rule 
by replacing Rule M by Rule M*
introduced below 
so that 
a suitable micro $K_4$-link $F_i$ from $G_{i-1}$ 
is chosen.

\vspace{0.2 cm} 
\noindent
{\bf Rule M*} (This rule is applied only when there is neither strong nor weak $K_4$-link $F$ from $G'$
which is a proper vertex-induced subgraph of $G$):

\vspace{-3 mm}

\begin{itemize}[itemsep=-1mm]
	\item Choose any vertex 
	$u$  in $G'$ such that 
	$N_G(u)\not\subseteq V(G')$.
	Such a vertex exists due to the condition that $G$ is connected and $G'$ is a proper vertex-induced subgraph of $G$. 
	
	\item Choose  a vertex $x\in N_G(u)\setminus V(G')$ such that 
	for some $w\in N_G(u)\cap  V(G')$,  $ux$ 
	(or its near 
	half-edge incident with $u$) 
	and $uw$ 
	(or its near half-edge
	incident with $u$) 
	are consecutive 
	on $\border({\cal R})$ 
	for some face ${\cal R}$ of $G$.
	Such vertices $x$ and $w$ exist
	due to the condition that 
	$G$ is $1$-plane
	and $N_G(u)\not\subseteq V(G')$.
	
	\item 
	If $ux$ is clean in $G$, 
	some $K_4$-subgraph $F$
	containing edge $ux$ is 
	arbitrarily chosen.
	If $ux$ is crossed by another edge 
	$st$ in $G$,  the subgraph 
	$F:=G[\{u,x,s,t\}]$ is chosen. 
\end{itemize}

Note that at the last step in Rule M*,  
the subgraph $F=G[\{u,x,s,t\}]$ chosen 
is indeed isomorphic to $K_4$
by Lemma~\ref{K_4}.
Since this rule is applied only when 
there is neither strong nor weak 
$K_4$-link from $G'$, 
we have $s,t\not\in V(G')$
and  $F$ is a micro $K_4$-link from $G'$.

When Rule M is replaced by Rule M*, 
the SWM-rule is strengthened to 
a new rule, called the SWM*-rule. 
Note that the SWM*-rule is 
only applied to a maximal $1$-plane 
graph in which 
each edge is contained in some $K_4$-subgraph.

A sequence  $\{G_i\}_{i\in \brk{0,N}}$ of subgraphs of $G$ is called 
a {\it $K_4$-extension sequence} 
of $G$ determined by the SWM*-rule
if $G_0\cong K_4$,  
$G_{N}=G$,  and for 
each $i\in \brk{N}$, 
$G_{i}=G[V(G_{i-1})\cup V(F_i)]$, 
where 
each $F_i$ is a 
$K_4$-link from $G_{i-1}$
determined 
by the SWM*-rule.
By the definition of the SWM*-rule, 
such a sequence exists for each 
maximal $1$-plane 
graph $G$ with $n(G)\ge 4$ and 
the property that 
each edge in $G$ is contained in some $K_4$-subgraph.


\begin{proposition}\label{exten0}
	Let $G$ be a maximal $1$-plane 
	graph with $n(G)\ge 4$ and 
	the property that 
	each edge in $G$ is contained in some $K_4$-subgraph.
	Then
	there exists 
	a {\it $K_4$-extension sequence} 
	$\{G_i\}_{i\in \brk{0,N}}$
	of $G$ determined by the SWM*-rule.
\end{proposition}

\section{Properties of $K_4$-extension sequence}\label{sec4}

Throughout this section, 
we assume that $G$ is a  maximal 1-plane graph with 
$n(G)\geq 4$ and the property that 
each edge in $G$ is contained in 
some $K_4$-subgraph.
By Proposition~\ref{exten0}, 
there exists 
a $K_4$-extension sequence 
$\{G_i\}_{i\in \brk{0,N}}$  
of $G$ determined by the SWM*-rule.
Thus, $G_0\cong K_4$, $G_N\cong G$, 
and for each $i\in \brk{N}$, 
$G_i=G[V(G_{i-1})\cup V(F_i)]$,
where 
$F_i$ is a $K_4$-link from $G_{i-1}$ 
determined by the SWM*-rule.
In this section, we will present  
some properties on
$\{G_i\}_{i\in \brk{0,N}}$.

\begin{lemma}\label{micro-aa}
	For any $i\in \brk{N}$, if 
	$F_i$ is a micro $K_4$-link from $G_{i-1}$,
then 
	$E(G_i)\setminus (E(G_{i-1})\cup E(F_i))\ne \emptyset$.
\end{lemma}

\begin{proof}
By the assumption, 
$F_i$ is chosen by Rule M*. 
Let $u$ be the only vertex in 
$V(F_i)\cap V(G_{i-1})$.
Then 
$F_i$ contains edge $ux$, where 
$x\in N_G(u)\setminus V(G_{i-1})$,
such that for some $w\in N_G(u)\cap V(G_{i-1})$, 
$ux$ 
(or its near half-edge
incident with $u$) 
and $uw$ 
(or its near half-edge
incident with $u$)
are consecutive 
on $\border({\cal R})$ 
for some face ${\cal R}$ of $G$.

\incase $uw$ is clean in $G$. 

If edge $ux$ is also clean in $G$, 
then $w$, $u$ and $x$ lie on the boundary of the same face in $G$. 
By Lemma~\ref{adj}, $w$ and $x$ are adjacent in $G$ 
and thus also adjacent in $G_{i}$ because $G_{i}$ is a vertex-induced subgraph of $G$. 

Now assume that edge $ux$ crosses 
edge $st$. 
By Rule M*,
$F_i$ is the subgraph $G[\{s, t, u, x\}]$
which is a $K_4$-subgraph by Lemma~\ref{K_4}.
Then, either $w, u, s$, or $w, u, t$ lie on $\border({\cal R})$. By Lemma~\ref{adj} , 
$w$ is adjacent to either $s$ or $t$ in $G_i$.
Thus, the conclusion holds in this case.

\incase $uw$ is crossed by 
edge $bb'$.


If $\{b,b'\}\not \subseteq V(G_{i-1})$,
then $F':=G[\{w,u,b,b'\}]$ 
is a $K_4$-subgraph by Lemma~\ref{K_4},
implying that $F'$ is 
either a strong or weak $K_4$-link from $G_{i-1}$,
contradicting the SWM*-rule 
as $F_i$ is a micro $K_4$-link from $G_{i-1}$. 
Thus, 
$\{b,b'\}\subseteq V(G_{i-1})$. 

If $ux$ is clean in $G$, 
then $\border({\cal R})$ contains $x$ 
and either $b$ or $b'$, implying 
$E(\{x\}, \{b,b'\})\ne \emptyset$
by Lemma~\ref{adj}. 
If $ux$ is crossed by 
some edge $st$, 
then $\border({\cal R})$ 
contains one vertex in $\{s,t\}$ 
and one vertex in $\{b,b'\}$,
implying that 
$E(\{s,t\}, \{b,b'\})\ne \emptyset$
by Lemma~\ref{adj}. 
The conclusion also holds in Case 2.

Hence the result holds.
\end{proof}

By Lemma \ref{micro-aa}, 
	if $F_i$ is a micro $K_4$-link from $G_{i-1}$,  then 
	$E(G_i)\setminus \big(E(G_{i-1})\cup E(F_i)\big)$ is a non-empty set. 
	If this set contains
	exactly one edge, namely $f_i$, 
	then $F_i$ is referred as 
	a {\it simple} micro $K_4$-link from $G_{i-1}$ and $f_i$  is referred to as the {\it extra} edge in $G_i$. 
	Otherwise, $F_i$ is said to be 
	{\it non-simple}.

\begin{lemma}\label{micro-sim}
	For any $i\in \brk{N}$, if 
	$F_i$ is a simple micro $K_4$-link from $G_{i-1}$,
	then  $i<N$ and $G_{i+1}$ is
	a strong or weak $K_4$-extension of  $G_i$.  
\end{lemma} 

\begin{proof}
	Let $f_i$ is the extra edge in $G_i$,
	i.e., the only edge in 
	$E(G_i)\setminus (E(G_{i-1})\cup E(F_i))$.
It follows that 
	$f_i$ is not contained in any 
	$K_4$-subgraph of $G_i$, implying that $i<N$.  
	By the assumption on $G$, $f_i$ 
	is contained in some $K_4$-subgraph $F$ of $G$. 
	Thus, $2\le |V(F)\cap V(G_i)|\le 3$, 
	implying that 
	$F$ is a strong or weak $K_4$-link from $G_i$. 
	By the SWM*-rule, $G_{i+1}$ must be a strong or weak $K_4$-extension of $G_i$. The result holds.
\end{proof}

In the following, we study the connectivity of graphs in 
the sequence $\{G_i\}_{i\in \brk{0,N}}$.
We will apply the following conclusion 
which can be verified easily. 

\begin{lemma}\label{2-con}
	Let $H_0$ and $H_1$ be  
	subgraphs of a graph $H$
	with $V(H_0)\cup V(H_1)=V(H)$,  $E(H_0)\cup E(H_1)=E(H)$
	and $|V(H_0)\cap V(H_1)|\ge 2$. 
	Assume that $H_0$ is $2$-connected
	and  $H_1$ is connected.
	Then, $H$ is $2$-connected 
	if and only if 
for each cut-vertex $w$ of $H_1$, 
	 $V(H')\cap V(H_0)\ne \emptyset$ for 
	each component $H'$ of $H_1-w$.
\end{lemma}

For each $i\in \brk{N}$, 
let $F_i'$ denote the subgraph of $G_i$ 
induced by $E(G_i)\setminus E(G_{i-1})$.
Clearly, $F'_i$ is a split graph
(i.e., a graph whose vertex set can be partitioned into a clique and an independent set), 
since $V(F_i)$ is a clique 
and $V(F_i')\setminus V(F_i)$ ($\subseteq V(G_{i-1})$)
is an independent set of $F_i'$.


\begin{lemma}\label{F_i}
Let $i\in \brk{N}$ and 
$e\in E(F_i)$.
For any cut-vertex $w$ of $F_i'-e$,
$w$ is contained in 
$V(F_i)\setminus V(G_{i-1})$ 
and each component of 
$F'_i-e-w$ contains vertices in $G_{i-1}$.
\end{lemma}

\begin{proof}
	Since $V(F'_i)$  can be partitioned into a clique 
	$V(F_i)$ and an independent set $V(F_i')\setminus V(F_i)$,
	any cut-vertex $w$ of $F_i'-e$ 
 must be contained in $V(F_i)\setminus V(G_{i-1})$.
	If $F_i$ is either a strong or weak $K_4$-link from $G_{i-1}$,
	it is clear that 
	each component of 
	$F'_i-e-w$ contains vertices in $G_{i-1}$.
	If $F_i$ is a micro $K_4$-link from $G_{i-1}$, the result also 
	follows directly from Lemma~\ref{micro-aa}. 
\end{proof}

\begin{lemma}\label{SWM}
	For any $i\in \brk{0,N}$, 
	$G_i$ is $2$-connected.
\end{lemma}

\begin{proof}
	Note that $G_0\cong K_4$ is $2$-connected. 
For any $i\in \brk{N}$,
	if $G_{i-1}$
	is $2$-connected, then by Lemmas~\ref{2-con} and~\ref{F_i}, 
	$G_i$ is also $2$-connected. 
	Hence the conclusion holds.
	\end{proof}

An edge $f$ in a 2-connected graph $H$ is said to be {\it removable} in $H$ 
if $H-f$ is still 2-connected.
Otherwise, it is  {\it non-removable}.

\begin{lemma}\label{SWM-a}
	For any  $i\in \brk{N}$,
	the following conclusions hold:
	
	\vspace{-3 mm}
	
	\begin{itemize}[itemsep=-1mm]
		\item[(1)] any edge $f$ in 
		$E(F_i)\setminus E(G_{i-1})$ is 
		removable  in $G_i$;
		
		\item[(2)] if $f\in E(G_{i-1})$
		is removable in $G_{i-1}$, then 
		it is  also removable in $G_{i}$;
		
		\item[(3)] 
		if $F_i$ is a non-simple micro $K_4$-link from $G_{i-1}$, 
		then any edge $f$ in $E(G_i)\setminus \big(E(G_{i-1})\cup E(F_i)\big)$ is  
		removable in $G_{i}$; and 
		
		\item[(4)] an edge $f$ in $E(G_i)\setminus \big(E(G_{i-1})\cup E(F_i)\big)$
		is non-removable in $G_{i}$ if and only if $F_i$ is a simple micro $K_4$-link from $G_{i-1}$ and  
		$f$ is the extra edge in $G_i$.
	\end{itemize}
\end{lemma}

\begin{proof}
	(1), (2) and (3) 
	can  be proved similarly as 
	Lemma~\ref{SWM} 
	by applying Lemmas~\ref{2-con}
	and~\ref{F_i}. 
	
	In order to prove (4), 
	by the result of (3), we need only to 
	consider the case that 
	$F_i$ is a simple micro $K_4$-link from $G_{i-1}$.
	Notice that in this case, 
	$E(G_i)\setminus \big(E(G_{i-1})\cup E(F_i)\big)$
	has a unique edge $f$, and this edge 
	is the extra edge in $G_i$
	and non-removable in $G_{i}$.
\end{proof}

\begin{lemma}\label{2-cut-a}
	For any $i\in \brk{N}$, if $\{s, t\}$ is a 
	cut-set of $G_i$ but not of $G_{i-1}$,
	then 
	\begin{enumerate}[itemsep=-1mm]
		\item[(1)] $F_i$ is  either a weak or micro $K_4$-link from $G_{i-1}$; and 
		\item[(2)] $\{s,t\}\subseteq V(G_{i-1})$
		or $F_i$ is a micro $K_4$-link from $G_{i-1}$  
		and $|\{s,t\}\cap V(G_{i-1})|=1$.
	\end{enumerate}
\end{lemma}

\begin{proof} (1) 
	Suppose that $F_i$ is a strong $K_4$-link from $G_{i-1}$. 
	Then $V(G_i)\setminus V(G_{i-1})=\{x\}$
	with $|\partial_{G_i}(x)|\ge 3$
	for some vertex $x$.
	As $G_{i-1}$ is $2$-connected by Lemma~\ref{SWM}, we have 
	$\{s,t\}\subseteq V(G_{i-1})$.
	Since $d_{G_i}(x)\ge 3$,
	if $\{s, t\}$ is a cut-set of $G_i$,
	it must be a
	cut-set of $G_{i-1}$,
	a contradiction. 
	The result holds.
	
	(2) Assume that $F_i$ is a weak or micro $K_4$-link from $G_{i-1}$. 
	Then $V(F_i)\setminus V(G_{i-1})$ 
	is a clique of $G_i$ with 
	size $2$ or $3$,
	implying that $G_i-\{s,t\}$ is connected when 
	$\{s,t\}\subseteq V(F_i)\setminus V(G_{i-1})$. 
	Thus, 
	$\{s,t\}\cap V(G_{i-1})\ne \emptyset$,
	and if $|\{s,t\}\cap V(G_{i-1})|=1$,
	then, $F_i$ must be a micro $K_4$-link from $G_{i-1}$. 
	The conclusion holds.
\end{proof}

	For any vertex $v$ (resp., edge $e$) in $G$,
	let $\od(v)$ (resp., $\od(e)$)
	denote the smallest integer $i\ge 0$  
	with $v\in V(G_i)$ (resp., 
	$e\in E(G_i)$). 
	Thus, $0\le \od(v)\le N$
	and $0\le \od(e)\le N$,
	and $\od(v_1v_2)
	=\max\{\od(v_1), \od(v_2) \}$
	for any edge $v_1v_2$ in $G$.
	Clearly, $\od(v)$
	and  $\od(e)$ depend
	on the sequence $\{G_i\}_{i\in \brk{0,N}}$.

\begin{lemma}\label{2-cut-b}
	Let $f$ be a non-removable  edge of $G_i$, where  $i\in \brk{N}$.
	Then, $i_1=\od(f)>0$ and 
	$G_{i_1}$ is obtained from $G_{i_{1}-1}$ by a simple micro $K_4$-extension in which 
	$f$ is an extra edge.
\end{lemma}

\begin{proof}
	Since $G_i-f$ is not $2$-connected, 
	$i>0$. 
	Let $i_1=\od(f)$. 
	Clearly, $i_1\le i$.
	If $f$ is removable in $G_{i_1}$, 
	$f$  is removable in $G_j$ for 
	all $j$ with $i_1\le j\le N$
	by applying Lemma~\ref{SWM-a} (2)
	repeatedly, 
	contradicting the given condition.	
	Thus, $f$ is non-removable in $G_{i_1}$, implying that $i_1>0$.
	By Lemma~\ref{SWM-a} (4),
	$F_{i_1}$ is a simple micro $K_4$-link
	from $G_{i_1-1}$ and $f$ is the extra edge in $G_{i_1}$. 
	The result holds.
\end{proof}

\begin{lemma}\label{od-1} 
	Let $xx', e,f$ be edges in $G$. 
	If $\od(e)<\od(f)<\od(x)$ and 
	 $xx'$ crosses $e$, then 
	$f$ is removable 
	in $G_{i}$, where $i=\od(x)$.
\end{lemma}

\begin{proof}
	Let $e$ be edge $ab$.
	Since $xx'$ crosses $e$, 
	$F:=G[\{x,x',a,b\}]\cong K_4$
	by Lemma~\ref{K_4}.
	Thus, $|V(F)\cap V(G_{i_1})|\ge 2$,
	where $i_1=\od(e)<i=\od(x)$.
	By the SWM*-rule, 
	$G_j$ is a strong or weak extension 
	from $G_{j-1}$
	for all $j\in \brk{i_1,i}$.
	
	Suppose that $f$ is non-removable in $G_i$. By Lemma~\ref{2-cut-b}, 
	$G_{i_2}$ is obtained from $G_{i_2-1}$ by a simple micro
	$K_4$-extension, where $i_2=\od(f)$,
	contradicting the previous conclusion, as 
	$i_1<i_2<i$. 
	Thus, $f$ is removable and 
	the result holds.
\end{proof}

Now, applying the preceding lemmas, we end this section 
with the following conclusion, 
which plays an important role in the proof of Lemma~\ref{strong}. 

\begin{proposition}
	\label{cut-bad}
	Assume that 
	$F_{i+1}$ is a strong $K_4$-link from  $G_{i}$ and 
	$x\in V(F_{i+1})\setminus V(G_i)$,
	where $i\in \brk{0,N-1}$.
	Then, none of the following situations can happen:
	
	\vspace{-3 mm}
	
	\begin{enumerate} [itemsep=-1mm]
		\item[(1)]  $G_i$ has distinct non-removable edges $f_1=a_1b_1$ and $f_2=a_2b_2$ which are crossed 
		by edges $xc_1$ and $xc_2$ respectively,
		where $c_1,c_2\in V(G_i)$, 
		as shown in Figure~\ref{Pr4.9-f1} (a); and 
		
	\item[(2)] $G_{i}$ has a $3$-cycle $abca$ 
	such that
	$ab$ is non-removable in $G_i$, 
	both $\{a, c\}$ and $\{b, c\}$ are cut-sets of $G_{i}$,  and 
	$xc$ crosses $ab$,
	as shown in Figure~\ref{Pr4.9-f1} (b).
	\end{enumerate}
\end{proposition}

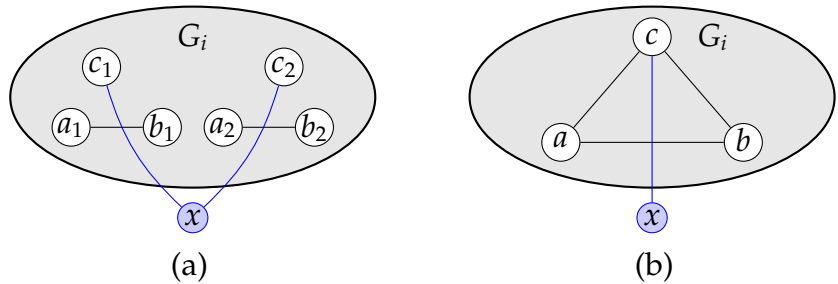
\begin{figure}[!ht]
	\centering
	
	\begin{tikzpicture}[scale=0.8]
		\tikzset{
			whitenode/.style={circle, draw=black, fill=white, minimum size=5mm,inner sep=2pt},
			whitenode2/.style={circle, draw=black, fill=white, minimum size=2mm},
			bluenode/.style={circle, draw=blue, fill=blue!20, minimum size=2mm},
			blueedge/.style={draw=blue},
			blackedge/.style={draw=black, line width=1.5pt}
		}
		\draw[thick,fill=gray!20] (0, 3) ellipse (3.0cm and 1.5cm);
		\node (G) at (0.,4) {$G_i$};	
		\node [bluenode] (x) at (0, 1) {};
		\node [whitenode] (a1) at (-2,2.5) {};
		\node [whitenode] (b1) at (-0.5, 2.5) {};
		\node [whitenode] (a2) at (0.5, 2.5) {};
		\node [whitenode] (b2) at (2, 2.5) {};
		\node [whitenode] (c1) at (-1.5, 3.5) {};
		\node [whitenode] (c2) at (1.5, 3.5) {};
		\node [] () at (0, 1) {$x$};
		\node [] () at (-2,2.5) {$a_1$};
		\node [] () at (-0.5, 2.5) {$b_1$};
		\node [] () at (0.5, 2.5) {$a_2$};
		\node [] () at (2, 2.5) {$b_2$};
		\node [] () at (-1.5, 3.5) {$c_1$};
		\node [] () at (1.5, 3.5) {$c_2$};
		\draw  (a1) to (b1);
		\draw  (a2) to (b2);
		\draw [blueedge,bend left =15] (x) to (c1);
		\draw [blueedge,bend right =15] (x) to (c2);
	\end{tikzpicture}
	\hspace{1 cm}
	\begin{tikzpicture}[scale=0.8]
		\tikzset{
			whitenode/.style={circle, draw=black, fill=white, minimum size=5mm,,inner sep=2pt},
			whitenode2/.style={circle, draw=black, fill=white, minimum size=2mm},
			bluenode/.style={circle, draw=blue, fill=blue!20, minimum size=2mm},
			blueedge/.style={draw=blue},
			blackedge/.style={draw=black, line width=1.5pt}
		}
		
		\node (G) at (1,4) {$G_i$};	
		\node [bluenode] (x) at (0, 1) {};
		\node [whitenode] (c) at (0,4) {};
		\node [whitenode] (a) at (-1.5,2.25) {};
		\node [whitenode] (b) at (1.5,2.25) {};
       \node () at (0, 1) {$x$};
		\node  () at (0,4) {$c$};
		\node  () at (-1.5,2.3) {$a$};
		\node  () at (1.5,2.3) {$b$};
		
		\scoped[on background layer]{
			\draw[thick,fill=gray!20] (0, 3) ellipse (3.0cm and 1.5cm);  
			\draw  (a) -- (b)--(c)-- (a);
            \draw  [blueedge] (x) -- (c);
}
		
	\end{tikzpicture}
	
	(a) \hspace{5.4 cm} (b)
	
	\caption{$F_{i+1}$ is a strong-$K_4$-link from $G_i$ 
		and $x\in V(F_{i+1})\setminus V(G_i)$}
	\label{Pr4.9-f1}
\end{figure}

\begin{proof}
	Suppose (1) happens. 
	For any $j=1,2$, 
	since $f_j$ is non-removable 
	in $G_i$, 
	by Lemma~\ref{2-cut-b},
	$f_j\notin  E(G_{i_j-1})$,
	where $i_j=\od(f_j)>0$, 
	and $G_{i_j}$ is obtained from $G_{i_j-1}$ by a simple micro 
	$K_4$-extension in which 
	$f_j$ is an extra edge. 
	Clearly, $i_1\ne i_2$. 
	Assume that $1\le i_1<i_2\le i$. 
	
	As $\od(x)=i+1$, 
	$1\le i_1<i_2\le i$ implies that  $\od(a_1b_1)<\od(a_2b_2)<\od(x)$. 
	By Lemma~\ref{od-1}, 
	$a_2b_2$ is removable in $G_{i}$, a contradiction to the given condition. 	
	Thus, (1) cannot happen. 
		
	Now suppose (2) happens. 
Since $ab$ is non-removable in $G_i$,  by Lemma \ref{2-cut-b},
	$ab\in E(G_{i_1})\setminus E(G_{i_1-1})$,
	where $0<i_1=\od(ab)\le i$, 
	and $G_{i_1}$ is a simple micro
	$K_4$-extension of $G_{i_1-1}$ 
	in which $ab$ is the extra edge.
	Thus, $\od(a)\ne \od(b)$.
	Assume that $\od(a)<\od(b)=i_1$.
	
	Now we establish the following claims
	to show that (2) cannot happen. 
	
	\setcounter{countclaim}{0}
	
	\inclaim $\{a,b,c,x\}$ is a clique of $G$.
	
	Since $xc$ crosses $ab$, 
	by Lemma~\ref{K_4}, 
	$\{a,b,c,x\}$ is a clique of $G$. 
	The claim holds.

	\inclaim $\od(a)<\od(b)\le \od(c)$.
	
	Suppose the claim fails. Then
	$\od(c)<\od(b)=i_1$. 
	By assumption, 
	we have  $\od(a)<\od(b)=i_1$,
	implying that 
	$\{a,c\}\subseteq V(G_{i_1-1})$.
	By Claim 1, 
	$F:=G[\{a,b,c,x\}]\cong K_4$.
	As $\od(x)=i+1>i_1=\od(b)$, 
	$|V(F)\cap V(G_{i_1-1})|=2$.
	By the SWM*-rule, 
	$G_{i_1}$ must be a
	strong or weak $K_4$-extension
	of $G_{i_1-1}$, 
	contradicting 
	the fact that 
	$G_{i_1}$ is a simple micro
	$K_4$-extension of $G_{i_1-1}$.
Thus, Claim 2 holds.

	\inclaim There exists an integer 
	$s$ with 
	$\od(c)\le s\le i$ such that 
	$\{b,c\}$ is a cut-set of $G_s$ 
	but not $G_{s-1}$, and 
	$G_s$ is not a 
	strong extension of $G_{s-1}$. 
	
	Note that $\od(b)\le \od(c)$ 
	by Claim 2.
	Since $\{b,c\}$ 
	is a cut-set of $G_i$,  
	there exists $s$ with 
	$\od(c)\le s\le i$
	such that $\{b,c\}$ is a cut-set of $G_s$ but not $G_{s-1}$.
	By Lemma~\ref{2-cut-a} (1), 
	$G_s$ is not a 
	strong extension of $G_{s-1}$.
	Thus, Claim 3 holds.
	
	\inclaim $\od(c)=s$. 

Suppose the claim fails. 
Then, by Claim 3, $\od(c)\le s-1$.
By Claim 2, $\od(a)<\od(b)\le \od(c)\le s-1$, 
implying that 
$\{b,c\}\subseteq V(G_{s-1})$.
	
	Since $F=G[\{a,b,c,x\}]\cong K_4$ with $|V(F)\cap V(G_{s-1})|=3$
	and $\od(x)=i+1$, 
	by the SWM*-rule, 
	$G_j$ must be a strong $K_4$-extension
	of $G_{j-1}$ for all 
	$j\in \brk{s,i+1}$,
	contradicting Claim 3.
	Thus, Claim 4 holds.
	
	\inclaim $\od(b)<s$ 
	and $G_s$ is a micro $K_4$-extension
	of $G_{s-1}$. 
	
	By Claim 3, $\{b,c\}$ is a cut-set of $G_s$ but not $G_{s-1}$. 
	By Claim 4, $\{b,c\}\not\subseteq V(G_{s-1})$. 
	Then, by Lemma~\ref{2-cut-a} (2), 
	$G_s$ is a micro $K_4$-extension
	of $G_{s-1}$ and $|\{b,c\}\cap V(G_{s-1})|=1$, 
	implying that $\od(b)\ne \od(c)$. 
	Thus, by Claim 2, $\od(b)<\od(c)=s$. 
	Claim 5 holds.

	By Claims 2 and 5, 
	$\od(a)<\od(b)\le s-1$. 
	By Claim 4, $\od(c)=s$. 
	Thus, 
$|\{a,b,c,x\}\cap V(G_{s-1})|=2$.
As $F=G[\{a,b,c,x\}]\cong K_4$
	and $\od(x)=i+1$, 
	by the SWM*-rule, 
	$G_j$ must be a strong or weak $K_4$-extension
	of $G_{j-1}$ for all $j=s, \dots, i+1$,
	contradicting Claim 5. 
	Hence (2) cannot happen.
\end{proof}

\section{An inequality involving 
order, size and nest numbers
}
\label{sec5}

Throughout this section, 
we assume that $G$ is a maximal 1-plane graph with $n(G)\geq 4$ 
and the property that each edge in $G$ is contained in a $K_4$-subgraph.
Let $\{G_i\}_{i\in \brk{0,N}}$ be a $K_4$-extension sequence of $G$ 
determined by the SWM*-rule. 
Thus, $G_0\cong K_4$ and $G_N\cong G$, 
and for each $i\in \brk{N}$, 
$G_i=G[V(G_{i-1})\cup V(F_i)]$,
where 
$F_i$ is a $K_4$-link from $G_{i-1}$ 
determined by the SWM*-rule.


Let $\Nest(H)$ be the set of nests in a $1$-plane graph $H$, and 
$\nest(H)=|\Nest(H)|$. 
In this section, we shall show that 
if $n(G)\ge 5$, then  
\equ{e5-1}
{
e(G)\geq \frac{7}{3}n(G)
+\frac{1}{3}\nest(G)-3. 
}
In order to prove (\ref{e5-1}), we shall first establish the inequality 
$e(G_{i+1})-e(G_i)\ge \frac 73 (n(G_{i+1})-n(G_i))+\frac 13|\Nest_{i+1}^*|$
for all $i\in \brk{0,N-1}$,
where  $\Nest_{i+1}^*:=
\Nest(G_{i+1})\setminus \Nest(G_i)$.
The proof of this inequality will be completed 
in Subsection 5.1
according to the three possible types of 
$K_4$-links $F_{i+1}$ from $G_i$.

\subsection{Lower bound for $e(G_{i+1})-e(G_i)$} 

\begin{lemma}
	\label{strong}
For any $i\in\brk{0,N-1}$,
if $F_{i+1}$ is a strong $K_4$-link from $G_i$, then
\equ{le-st-e1}
{
e(G_{i+1})-e(G_i)\ge \frac 73 +\frac 13 
|\Nest_{i+1}^*|.
}
\end{lemma}

\begin{proof}
Since $F_{i+1}$ is a strong $K_4$-link from $G_i$, 
$V(F_{i+1})\setminus V(G_{i})$ has a unique vertex, say $x$.	
Let $p$ (resp., $\ell$) be the number of clean (resp.,  crossing)  edges
in $\partial_{G_{i+1}}(x)$.
Then, $\ell+p=d_{G_{i+1}}(x)
 \geq d_{F_{i+1}}(x)=3$. 

A nest in $\Nest_{i+1}^*$ is said to be {\it new}. 
Let ${\cal N}=\langle \{s,t\}, \{\alpha_1, \alpha_2\}\rangle$ be a new nest,
where $s$ and $t$ are vertices in $G_{i+1}$
and $\alpha_1$ and $\alpha_2$ are crossing points.
Then, either $x\in \{s, t\}$ 
or $\border (\cal N)$  
contains at least one far half-edge incident with $x$. $\setn$ is referred to as {\it type A}
if $x\in \{s, t\}$, and {\it type B}
otherwise,
as shown in Figure~\ref{new-nest}.
Thus,  $\Nest_{i+1}^*$ 
can be partitioned into 
two subsets, i.e., the set $\Nest^*_A$
of new nests in type A 
and the set $\Nest^*_B$ 
of new nests in type B.

Note that, as $d_{G_{i+1}}(x)\geq 3$,
the strong $K_4$-link $F_{i+1}$ 
does not destroy any nest in $G_{i}$ because $x$ can not lie in any nest of $G_{i}$.
Thus, $\nest(G_{i+1})=
\nest(G_i)+|\Nest^*_A|+|\Nest^*_B|$, and
(\ref{le-st-e1}) is transferred into 
the following inequality: 
\equ{le-st-e2}
{
	\ell+p\ge \frac 73+
	\frac{1}3 (|\Nest^*_A|+|\Nest^*_B|).
}

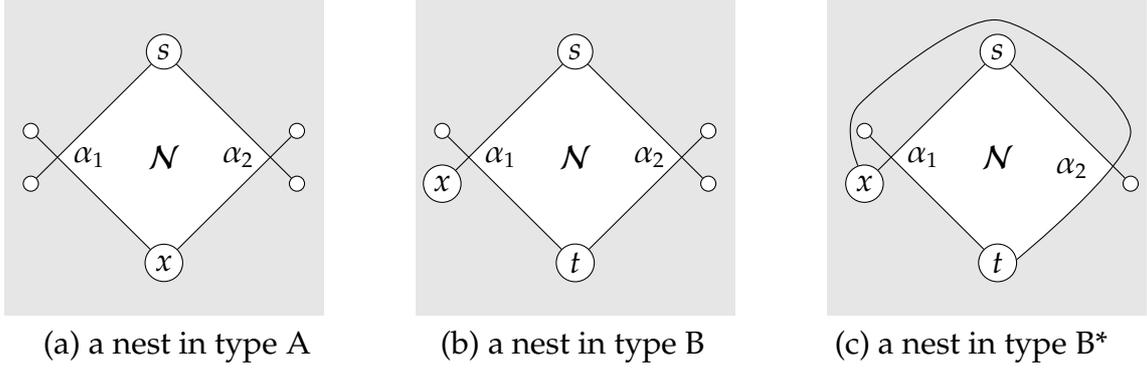
\begin{figure}[h]
  \centering
\begin{tikzpicture}[scale=0.7]

\tikzset{vertex/.style={circle,draw=black,fill=white,inner sep=2pt},
  blackedge/.style={thick}
}


\fill[gray!20] (-3,-3) rectangle (3,3);
\fill[white] (-2,0) -- (0,-2) -- (2,0) -- (0,2) -- cycle;
\node[vertex] (s) at (0,2) {$s$};
\node[vertex] (x) at (0,-2) {$x$};
\node at (-1.4, 0) {$\alpha_1$}; 
\node at (1.4, 0) {$\alpha_2$};  

\node[vertex] (x1) at (-2.5,0.5) {};
\node[vertex] (x2) at (-2.5,-0.5) {};
\node[vertex] (y1) at (2.5,0.5) {};
\node[vertex] (y2) at (2.5,-0.5) {};
\draw (s) -- (x2);
\draw (s) -- (y2);
\draw (x) -- (x1);
\draw (x) -- (y1);

\node at (0,0) {$\mathcal{N}$};


\end{tikzpicture}
\hspace{1 cm} 
\begin{tikzpicture}[scale=0.7]
	
\tikzset{vertex/.style={circle,draw=black,fill=white,inner sep=2pt},
  blackedge/.style={thick}
}


\fill[gray!20] (-3,-3) rectangle (3,3);
\fill[white] (-2,0) -- (0,-2) -- (2,0) -- (0,2) -- cycle;
\node[vertex] (s) at (0,2) {$s$};
\node[vertex] (t) at (0,-2) {$t$};
\node at (-1.4, 0) {$\alpha_1$}; 
\node at (1.4, 0) {$\alpha_2$};  

\node[vertex] (x1) at (-2.5,0.5) {};
\node[vertex] (x2) at (-2.5,-0.5) {$x$};
\node[vertex] (y1) at (2.5,0.5) {};
\node[vertex] (y2) at (2.5,-0.5) {};
\draw (s) -- (x2);
\draw (s) -- (y2);
\draw (x) -- (x1);
\draw (x) -- (y1);

\node at (0,0) {$\mathcal{N}$};
\end{tikzpicture}
\hspace{1 cm} 
\begin{tikzpicture}[scale=0.7]
	
	\tikzset{vertex/.style={circle,draw=black,fill=white,inner sep=2pt},
		blackedge/.style={thick}
	}

	  \begin{scope}[on background layer={color=gray!20}]
      \fill (-3.2,-3) rectangle (2.8,3);
    \end{scope}
	
	\node[vertex] (s) at (0,2) {$s$};
	\node[vertex] (t) at (0,-2) {$t$};
	\node at (-1.4, 0) {$\alpha_1$}; 
	\node at (1.4, -0.2) {$\alpha_2$}; 
	
	\node[vertex] (lef) at (-2.5,0.5) {};
	\node[vertex] (x) at (-2.5,-0.5) {$x$};
	
	\node[vertex] (y2) at (2.5,-0.5) {};
	\draw (s) -- (x)[spath/save=sx];
	\draw (s) -- (y2)[spath/save=sy2];
	\draw (lef) -- (t)[spath/save=left];
	\draw plot [smooth, tension=0.8] coordinates {(t.10) (2.5,0.5)  (0,2.6)  (-2.6, 1)  (x.110)}[spath/save=tx];

  \scoped[on background layer,
 spath/split at intersections={tx}{sy2},
  spath/remove components={tx}{2},
  spath/remove components={sy2}{2}] 
{\fill[white]
  (t) [spath/use={tx}]
  -- (spath cs:{tx} 0) [spath/use={sy2, weld}]
--(s.center)
--(-2,0)
--(t.center)
-- cycle;}

	\node at (0,0) {$\mathcal{N}$};

\end{tikzpicture}

(a) a nest in type A
\hspace{1.5 cm} 
(b) a nest in type B
\hspace{1.5 cm} 
(c) a nest in type B*

\caption{Two types of nests in $\Nest(G_{i+1})\setminus \Nest(G_i)$,
where a nest $\setn$ in type B* is 
a special one in type B whose boundary $\border(\setn)$ has two far half-edges incident with $x$}
\label{new-nest}
\end{figure}

A nest $\setn\in \Nest_B^*$
is referred to as 
type B* if $\border(\setn)$ 
has two far half-edges incident with $x$,
as shown in Figure~\ref{new-nest} (c). 
Let $\Nest^*_{B^*}$ be the set of 
new nests in type B*. 

\begin{claim}\label{AB-typed}
	$p\ge 2$, 
	$|\Nest^*_A|\le \ell$ and 
	$|\Nest^*_B|\le 2\ell-|\Nest^*_{B^*}|$.
	In particular, 
	$|\Nest^*_A|\le \ell-1$ 
	if $1\le \ell \le 2$.
\end{claim}

\begin{proof}
	It follows from Lemma~\ref{clean-e} that $p \geq 2$. 
	
	For any 
	$\setn\in \Nest^*_A$, 
	$\border(\setn)$ occupies exactly two consecutive near half-edges incident with $x$, as illustrated in  Figure~\ref{new-nest} (a).
	Meanwhile, a near half-edge incident with $x$ belongs to the boundaries of at most two distinct $A$-typed nests. Hence, $|\Nest^*_A|\le \ell$. 
	Since each nest in $\Nest^*_A$ 
	corresponds to two crossings, 
	$|\Nest^*_A|=0$
	when $\ell=1$. 
	If $\ell=2$, then $|\Nest^*_A|\le 1$, 
	 otherwise $G_{i+1}$ would have multiple edges.
	
	For any $\setn\in \Nest^*_B$, 
	  $\border(\setn)$ contains at least one of the far half-edges incident with $x$. 
	  Meanwhile, any far half-edge incident with $x$  belongs to the boundaries of at most two distinct  nests in $\Nest^*_B$.  
	  Therefore,  $|\Nest^*_B|\le 2\ell-|\Nest^*_{B^*}|$.
	  Claim~\ref{AB-typed} holds.
\end{proof}

\begin{claim}\label{AB-ok}
The inequality of (\ref{le-st-e2})
holds if  either $p\ge 3$ or $|\Nest^*_A|+|\Nest^*_B|\le 3\ell-1$.
\end{claim} 

\begin{proof}
Note that $p\geq 2$ by Claim~\ref{AB-typed}. 
Thus, (\ref{le-st-e2}) holds if $|\Nest^*_A|+|\Nest^*_B|\le 3\ell-1$. 
Then, this claim follows directly 
from Claim~\ref{AB-typed}. 	
\end{proof}

\begin{claim}\label{l51-cl0}
	If $p=2$ and $\ell\le 2$, then 
	the inequality of (\ref{le-st-e2}) 
	holds.
\end{claim} 

\begin{proof}
	Assume that $p=2$ and $\ell \le 2$.
	By Claims~\ref{AB-typed}, 
	$|\Nest^*_A|\le \ell-1$
	and $|\Nest^*_B|\le 2\ell$.
	Then, the claim follows directly. 
\end{proof} 

By Claims~\ref{AB-typed},~\ref{AB-ok}
and~\ref{l51-cl0}, 
it remains to consider the case 
$p=2$, $\ell\ge 3$, 
$|\Nest^*_A|=\ell$ and $|\Nest^*_B|=2\ell$.
	Let $\partial_{G_{i+1}}(x)=
	\{h_1,h_2\}\cup \{e_j: j\in \brk{\ell}\}$,
	where $h_1$ and $h_2$ are the only
	edges in  $\partial_{G_{i+1}}(x)$ 
	which are clean in $G_{i+1}$.

\begin{claim}\label{claim5.2}
	There exists a face ${\cal R}$ in $G_i$ that is bounded by 
	an $\ell$-cycle
	$C=v_{1}v_2\cdots v_{\ell}v_1$,
	consisting of clean edges in $G_i$,  such that $x$ lies in the interior of ${\cal R}$ and 
	for each $j\in \brk{\ell}$,
	some edge in 
 $\partial_{G_{i+1}}(x)$, say $e_j$, 
crosses $f_j$ , where $f_1:=v_{\ell}v_1$
	 and $f_j:=v_{j-1}v_j$,
	 as shown in Figure~\ref{strong-aa}. 
\end{claim}

\begin{figure}[!h]
	\centering
	\begin{tikzpicture}[scale=0.8]
	
		\foreach \i/\label in {22.5/3, 67.5/2, 112.5/1, 155.5/l, -67.5/5, 337.5/4} {
			\node[scale =0.8, draw, circle, fill=white, inner sep=1pt] (v_{\label}) at ($(0,0) + (\i:3)$) {\small $v_{\label}$};
		}
		\node[scale =0.5, draw, circle, fill=white, inner sep=1pt] (v_{l-1}) at ($(0,0) + (202.5:3)$) {\small $v_{l-1}$};
		\node[scale =0.5, draw, circle, fill=white, inner sep=1pt] (v_{l-2}) at ($(0,0) + (-112.5:3)$) {\small $v_{l-2}$};
		\draw[line width=1.5pt] (v_{5}) -- (v_{4}) -- (v_{3}) -- (v_{2}) -- (v_{1}) -- (v_{l}) -- (v_{l-1}) -- (v_{l-2});
		\draw[decorate, decoration={
			markings,
			mark=between positions 0.2 and 1 step 12pt with {\node[circle, fill=black, minimum size=0pt, inner sep=0.6pt]{};}
		}] (v_{5}) -- (v_{l-2});
		
		\foreach \x in {-45,0,45,...,225}{
			\draw[-, blue] (0,0) -- (\x:3.5);
		}
	
		\node[draw, circle, minimum size=2mm, fill=blue!30] (x) at (0,0) {$x$};
		
		\foreach \i/\label in {35/3, 80/2, 125/1, 170/l, 305/5, 350/4} {
			\node[scale=0.7] () at ($(0,0) + (\i:3.2)$) {$f_{\label}$};
			\node[scale=0.6] () at ($(0,0) + (\i:1.75)$) {$e_{\label}$};
		}
		\node[scale=0.6] () at ($(0,0) + (240:1.75)$) {$e_{l-1}$};
		\foreach \i/\label in {22.5/3, 67.5/2, 112.5/1, 155.5/l, 202.5/{l-1}, 337.5/4} {
			\node[scale=0.7] () at ($(0,0) + (\i:2.25)$) {$\setn_{\label}$};
		}
		
		\foreach \i in {-0.8, 0, 0.8} {
			\fill (\i, -2) circle (1.0pt);
		}
	\end{tikzpicture}

	\caption{A local structural description for Claim~\ref{claim5.2}, where $h_1$ and $h_2$ are omitted}
	\label{strong-aa}
\end{figure}
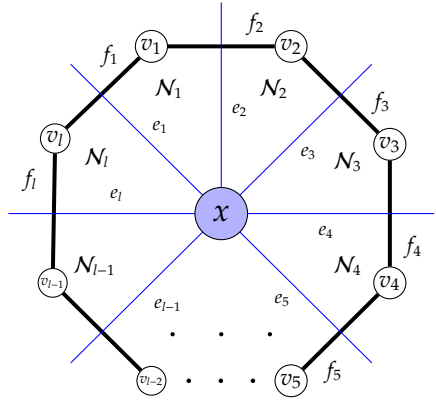

\begin{proof}
Without loss of generality, assume that
edges $e_1, e_2,\dots,e_{\ell}$ 
are in clockwise cyclic order.
Note that $|\Nest^*_A|=\ell$ and 
each nest in $\Nest^*_A$ contains 
two near half-edges incident to $x$.
Thus,  for each $k\in \brk{\ell}$,
there is a nest $\setn_k\in \Nest^*_A$
containing  the corner $\langle e_k x e_{k+1}\rangle$ (subscripts modulo $\ell$).
Note that $x$ is a supporting vertex of 
each $\setn_k$.
Let $v_k$ be the other supporting vertices of $\setn_k$.
Together with the property that 
$\setn_{k-1}$ and $\setn_k$ shares a near half-edge of $e_k$ in $\partial_{G_{i+1}}(x)$, 
we conclude that $v_{k-1}$ and $v_k$ 
are adjacent in $G_i$. 
Let $f_k:=v_{k-1}v_k$. 
Clearly, $f_k$ crosses $e_k$,
as shown in Figure~\ref{strong-aa}.

By the definition of nests, $\bigcup\limits_{k=1}^{\ell}\setn_{k}$ forms a face of $G_{i}$ bounded by $C$: $v_1v_2\cdots v_{\ell}v_1$ 
(see Figure  \ref{strong-aa}),
where each edge $v_{k-1}v_k$ 
is clean in $G_i$.
Thus,  Claim~\ref{claim5.2} holds.
\end{proof}

\begin{claim}\label{claim5.3}
	There are exactly $\ell-2$ edges in 
	$\{e_i: i\in \brk{\ell}\}$ 
	each of which 
	is incident with some vertex on $C$. 
\end{claim} 

\begin{proof}
By Claim~\ref{claim5.2}, for each $k\in \brk{\ell}$, 
$e_k$ crosses $v_{k-1}v_k$,
implying that 
$\{v_{k-1}, v_k\}\subseteq N_{G_{i+1}}(x)$ by Lemma~\ref{K_4}.
Hence $\{v_j:  j \in \brk{\ell}\}\subseteq  N_{G_{i+1}}(x)$. 

Note that $h_1$ and $h_2$ are the only edges in $\partial_{G_{i+1}}(x)$ which are clean in $G_{i+1}$.
Then,  for each $r\in [2]$, 
$h_r$  must be an edge 
within ${\cal R}$ 
which joins $x$ to some vertex  in 
$\{v_j: j\in \brk{\ell}\}$. 

For any $k\in \brk{\ell}\setminus  
\{s_1,s_2\}$, as $v_k\in N_{G_{i+1}}(x)$
and 
$\partial_{G_{i+1}}(x)=
\{h_1,h_2\} \cup \{e_j: j\in \brk{\ell}\}$,
 some edge 
in $\{e_j: j\in \brk{\ell}\}$
must be incident with $v_k$. 
Thus, the claim holds.
\end{proof}

\begin{claim}\label{claim5.4}
For any $k\in \brk{\ell}$, 
if $e_k$ is incident with some vertex in
$\{v_j: j\in \brk{\ell} \}$, then $f_k$ is non-removable in $G_i$. 
\end{claim}

\begin{proof}
Suppose that $e_1$ is incident with some vertex $v_j$. 
Since $e_1$ is not clean in $G_{i+1}$, 
$j\notin \{1,\ell\}$, 
as shown in Figure~\ref{cl5.4-p1}. 
Then, every path in $G_i-f_1$
connecting $v_1$ and $v_{\ell}$
contains vertex $v_j$,
implying that 
$v_j$ is a cut-vertex of $G_i-f_1$.
Thus, $f_1$ is non-removable
in $G_i$.
\end{proof}

\begin{figure}[!h]
	\centering
	\begin{tikzpicture}[scale=0.8]
		
		\foreach \i/\label in {22.5/3, 67.5/2, 112.5/1, 155.5/l, -67.5/5, 337.5/4} {
			\node[scale =0.8, draw, circle, fill=white, inner sep=1pt] (v_{\label}) at ($(0,0) + (\i:3)$) {\small $v_{\label}$};
		}
		\node[scale =0.5, draw, circle, fill=white, inner sep=1pt] (v_{l-1}) at ($(0,0) + (202.5:3)$) {\small $v_{l-1}$};
		\node[scale =0.5, draw, circle, fill=white, inner sep=1pt] (v_{l-2}) at ($(0,0) + (-112.5:3)$) {\small $v_{l-2}$};
		\draw[line width=1.5pt] (v_{5}) -- (v_{4}) -- (v_{3}) -- (v_{2}) -- (v_{1}) -- (v_{l}) -- (v_{l-1}) -- (v_{l-2});
		\draw[decorate, decoration={
			markings,
			mark=between positions 0.2 and 1 step 12pt with {\node[circle, fill=black, minimum size=0pt, inner sep=0.6pt]{};}
		}] (v_{5}) -- (v_{l-2});
		
		\foreach \x in {-45,0,45,...,90}{
			\draw[-, blue] (0,0) -- (\x:3.5);
		}
	
		\foreach \x in {180,225}{
		\draw[-, blue] (0,0) -- (\x:3.5);
	}

		\node[draw, circle, minimum size=2mm, fill=blue!30] (x) at (0,0) {$x$};
		
		\foreach \i/\label in {35/3, 80/2, 125/1, 170/l, 215/{l-1}, 305/5, 350/4} {
			\node[scale=0.7] () at ($(0,0) + (\i:3.2)$) {$f_{\label}$};
			\node[scale=0.6] () at ($(0,0) + (\i:1.75)$) {$e_{\label}$};
		}
		
		\foreach \i/\label in {22.5/3, 67.5/2, 112.5/1, 155.5/l, 202.5/{l-1}, 337.5/4} {
			\node[scale=0.7] () at ($(0,0) + (\i:2.25)$) {$\setn_{\label}$};
		}
		
		\foreach \i in {-0.8, 0, 0.8} {
			\fill (\i, -2) circle (1.0pt);
		}
	
\draw[blue] plot [smooth, tension=1.2] coordinates {(x.150)  (-2.2,3.1) (2.5, 3.1) (3.6, 0.0)   (v_{4}.15)};
	\end{tikzpicture}

\caption{$e_1$ is incident with $v_4$
}
	\label{cl5.4-p1}
\end{figure}
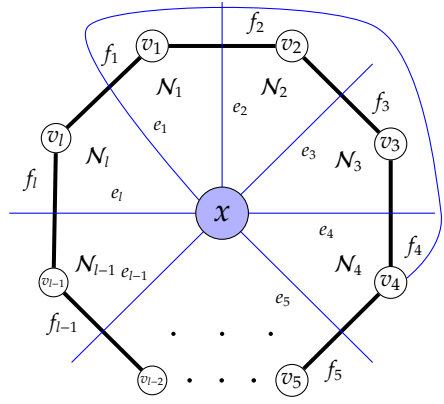

\begin{claim}\label{claim5.5}
If $p=2$,   then $\ell\le 3$. 
\end{claim} 

\begin{proof}
Suppose that $\ell\ge 4$.
	By Claim~\ref{claim5.3}, there are at least two distinct edges 
	$e_{s_1}$ and $e_{s_2}$, where 
	$1\le s_1<s_2\le \ell$, 
	each of which is incident with some 
	 vertex on $C$.
By Claim~\ref{claim5.4}, 
$f_{s_1}$ and $f_{s_2}$  are non-removable in $G_i$.

Note that $e_{s_j}\in \partial_{G_{i+1}}(x)$ 
crosses $f_{s_j}$ for $j=1,2$.
However, by Proposition~\ref{cut-bad} (1), this situation cannot happen, 
a contradiction.  
Hence this claim holds.
\end{proof}


\begin{figure}[!h]
	\centering
	\begin{tikzpicture}[bezier bounding box]
		\tikzset{vertex/.style={circle,draw=black,fill=white,inner sep=2pt}, vertexb/.style={circle,draw=black,fill=blue!30,inner sep=2pt},
			blackedge/.style={line width=1.5pt},  blueedge/.style={blue, line width=1.5pt}, blueedge1/.style={blue}
		}

		\node[vertex] (v1)  at  (-2, 0) {$v_1$};
		\node[vertex] (v2)  at  (2, 0) {$v_2$};
		\node[vertex] (v3)  at  (0, 3) {$v_3$};
		\node[vertexb](x)   at  (0,1.2) {$x$};
		\node[vertex](z1)  at  (-1.8,1.8) {$z_1$};
		\node[vertex](z2)  at  (1.8,1.8) 
		{$z_2$};
		
		\node[scale=1] ()  at (-0.2, 0.5) 
		{$e_2$};
		\node[scale=1] ()  at (0.5, -0.3) 
		{$f_2$};
		
		\node[scale=1] ()  at (0.7, 1.2) 
		{$e_3$};
		\node[scale=1] ()  at (2,  0.9) 
		{$f_3$};
		
		\node[scale=1] ()  at (-0.7, 1.2) 
		{$e_1$};
		\node[scale=1] ()  at (-2,  0.9) 
		{$f_1$};
		
		\draw[blackedge] (v1) -- (v2)--(v3)--(v1); 
		\draw[blueedge1] (x) -- (z1); 
		\draw[blueedge1] (x) -- (z2); 
		\draw[red!40] (x) -- (v1); 
		\draw[red!40] (x) -- (v2); 
		
		
		
		\draw[blueedge1] plot [smooth,tension=1] coordinates {(x.-80)  (0,0)  (-2.4,-0.3)  (-2.2,2)  
			(v3.200)};
		
		
	\end{tikzpicture}
	
	\caption{$\ell=3$}
	\label{strong-f}
\end{figure}

\begin{claim}\label{le51-cl2}
If $p=2$, then $\ell\ne 3$.
\end{claim}

\begin{proof}
Suppose that $p=2$ and $\ell=3$. 
By Claim~\ref{claim5.3},  exactly one 
edge in $\{e_1,e_2,e_3\}$ 
has an end-vertex on $C$. Assume that $e_2$ 
is the edge. 
Clearly, $e_2$ is incident with $v_3$
and crosses edge $v_1v_2$. 
Then, each edge in $\{e_1,e_3\}$ 
joins $x$ to a vertex 
which is not on $C$,
as shown in Figure~\ref{strong-f}. 

By Claim~\ref{claim5.4}, $f_2=v_1v_2$ 
is non-removable in $G_i$.
It is clear that $\{v_1,v_3\}$ 
is a cut-set of $G_i$, 
as $z_1$ and $z_2$ are in different components of the subgraph $G_i-\{v_1,v_3\}$. 
Similarly, $\{v_2,v_3\}$ 
is also a cut-set of $G_i$.
However, by Proposition~\ref{cut-bad} (2), this situation cannot happen,
a contradiction.  The claim holds.
\end{proof}

Inequality (\ref{le-st-e2})  follows from
By Claims~\ref{AB-ok}, ~\ref{l51-cl0},
~\ref{claim5.5} and~\ref{le51-cl2}.
Lemma~\ref{strong} holds.
\end{proof} 


\begin{lemma}
	\label{weak}
For any $i\in\brk{0,N-1}$, 
if $F_{i+1}$ is a weak $K_4$-link from $G_i$, then 
\equ{le5.2-e1}
{
 e(G_{i+1})-e(G_i)
 \geq \frac{14}{3}
 +\frac{1}{3}|\Nest_{i+1}^*|,
}
where 
the equality holds if and only if 
	$e(G_{i+1})-e(G_i)=5$ and 
	$|\Nest_{i+1}^*|=1$.
\end{lemma}

\begin{proof}
Let $V(F_{i+1})=\{x, y, v_1, v_2\}$, 
where $v_1, v_2\in V(G_i)$
and $x,y\in V(G_{i+1})\setminus V(G_i)$. 
Let $E_x=\partial_{G_{i+1}}(x)\setminus \{xy\}$ and 
$E_y=\partial_{G_{i+1}}(y)\setminus \{xy\}$.
Obviously, 
$|E_x|\ge 2$, $|E_y|\ge 2$
and $e(G_{i+1})-e(G_i)=|E_x|+|E_y|+1$. 
The following claim holds directly. 


\setcounter{claim}{0}

\begin{claim}\label{cl5.2.0}
(\ref{le5.2-e1}) is equivalent to $|\Nest_{i+1}^*|\le 
3(|E_x|+|E_y|-4)+1$,
and 
the conclusion holds  when 
$|\Nest_{i+1}^*|\le 1$.

\end{claim}


\begin{claim}\label{cl5.2.1}
	$G$ has no crossings between any edge in $E_x\cup E_y$ and any
	edge in $E(G_{i})$.
\end{claim}

\begin{proof}
	Suppose, to the contrary, that 
	some edge in $E_x\cup E_y$, say $xu$,  crosses
	an edge $u'v'$ of $G_{i}$ in $G_{i+1}$. By Lemma~\ref{K_4}, 
	$F':=G[\{x, u, u', v'\}]\cong K_4$.
	Thus,  $F'$ is a  strong $K_4$-link from $G_{i}$, contradicting the assumption 
	that $F_{i+1}$ is a weak $K_4$-link from $G_i$ chosen by the SWM*-rule. 
	Claim~\ref{cl5.2.1} holds.
\end{proof}

\begin{claim}\label{cl5.2.2}
If $x$ and $y$ are in the 
different faces ${\cal R}$ and ${\cal R}'$ of $G_i$,
then $\Nest_{i+1}^*=\emptyset$
and the inequality of (\ref{le5.2-e1}) 
is strict.
\end{claim}

\begin{proof}
By Claim~\ref{cl5.2.1}, 
all edges in $E_x$ lies in the interior of 
${\cal R}$ and all edges in $E_y$
in the interior of ${\cal R'}$,
as shown in Figure \ref{weak-bb}.
Since $xy\in E(G_{i+1})$, 
and $x$ and $y$ are in different faces of $G_i$, 
$xy$ must cross some edge $st$ of  $G_i$, 
which is on both $\border({\cal R})$ and  $\border({\cal R}')$.
By Claim~\ref{cl5.2.1}, 
no edge in $E_x\cup E_y$ 
crosses any edge in $E(G_i)$, 
implying that 
 $\Nest_{i+1}^*=\emptyset$.
Claim~\ref{cl5.2.2} follows from 
Claim~\ref{cl5.2.0}.

\end{proof}

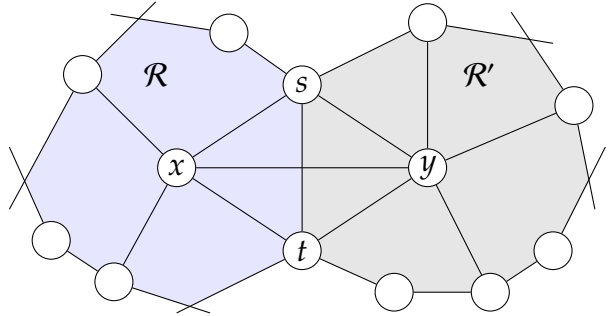
\begin{figure}[!ht]
	\centering
	
	\begin{tikzpicture}[scale=0.55,whitenode/.style={circle, draw=black, fill=white, minimum size=5mm, inner sep=1pt},
		pickshing/.style={draw=none, fill=gray!20}, shadow_blue/.style={fill=blue!10, draw=none}]
		\tikzset{edge/.style={}}
		
		\node [style=whitenode] (0) at (0, 2) {$s$};
		\node [style=whitenode] (1) at (0, -2) {$t$};
		\node [style=whitenode] (2) at (-1.75, 3.25) {};
		\node [style=whitenode] (3) at (-5.25, 2.25) {};
		\node [style=whitenode] (4) at (-6, -1.75) {};
		\node [style=whitenode] (5) at (-4.5, -2.75) {};
		\node [style=none] (6) at (-3, -3.5) {};
		\node [style=none] (8) at (-7, 0.5) {};
		\node [style=none] (9) at (-7, -1) {};
		\node [style=none] (10) at (-4.575, 3.7) {};
		\node [style=none] (11) at (-3.5, 4) {};
		\node [style=none] (12) at (-2.2, -3.5) {};
		\node [style=whitenode] (13) at (-3, 0) {$x$};
		\node [style=whitenode] (14) at (3, 0) {$y$};
		\node [style=whitenode] (15) at (3, 3.5) {};
		\node [style=none] (16) at (5, 3.75) {};
		\node [style=none] (17) at (6, 3) {};
		\node [style=whitenode] (18) at (6.5, 1.5) {};
		\node [style=none] (19) at (7.25, 0.5) {};
		\node [style=none] (20) at (7, -1) {};
		\node [style=whitenode] (21) at (2.2, -3) {};
		\node [style=whitenode] (22) at (4.5, -3) {};
		\node [style=whitenode] (23) at (6, -2) {};
		\node  (24) at (-3.5, 2.25) {${\cal R}$};
		\node(25) at (4.25, 2.25) {${\cal R}'$};
		\node [style=none] (26) at (-3.925, 3.6) {};
		\node  (27) at (-2.7, -3.35) {};
		\node [style=none] (28) at (5.425, 3.1) {};
		\node [style=none] (29) at (6.875, -0.325) {};
		\node (30) at (-6.625, -0.325) {};
		\draw[edge] (0) to (1);
		\draw[edge] (1) to (6.center);
		\draw[edge] (8.center) to (4);
		\draw[edge] (4) to (5);
		\draw[edge] (3) to (9.center);
		\draw[edge] (3) to (11.center);
		\draw[edge] (10.center) to (2);
		\draw[edge] (2) to (0);
		\draw[edge] (13) to (14);
			
		\scoped[on background layer]{
			\draw[style={pickshing}] (0.center) to (1.center) to (21.center) to (22.center)  to (23.center) to (29.center) to (18.center)  to (28.center)   to (15.center)   to cycle;                             \draw [style={shadow_blue}] (1.center)			 to (27.center)			 to (5.center)			 to (4.center)			 to (30.center)			 to (3.center)			 to (26.center)			 to (2.center)			 to (0.center)			 to cycle
			; }

		\draw[edge] (3) to (13);
		\draw[edge] (13) to (5);
		\draw[edge] (13) to (1);
		\draw[edge] (13) to (0);
		\draw[edge] (1) to (21);
		\draw[edge] (21) to (22);
		\draw[edge] (22) to (23);
		\draw[edge] (23) to (19.center);
		\draw[edge] (20.center) to (18);
		\draw[edge] (18) to (16.center);
		\draw[edge] (17.center) to (15);
		\draw[edge] (15) to (0);
		\draw[edge] (14) to (0);
		\draw[edge] (14) to (1);
		\draw[edge] (14) to (22);
		\draw[edge] (14) to (15);
		\draw[edge] (18) to (14);
		\draw[edge] (5) to (12.center);
		
	\end{tikzpicture}

	\caption{The case that  $x$ and $y$ lie in  distinct  faces  of $G_i$}
	\label{weak-bb}
\end{figure}

By Claim~\ref{cl5.2.0}, 
it suffices to consider  the case that $|\Nest_{i+1}^*|\ge 2$. 
Then, by Claim~\ref{cl5.2.2}, we may assume that 
both $x$ and $y$ lie in a
face ${\cal R}$ of $G_{i}$.

\begin{claim}\label{cl5.2.3}
If $|\Nest_{i+1}^*|\ge 2$ and 
$x$ and $y$ lie in the some
face ${\cal R}$ of $G_{i}$,  
then $|\Nest_{i+1}^*|\le 3$
and $|E_x|+|E_y|\ge 8$,
and hence the inequality of (\ref{le5.2-e1}) 
is strict.
\end{claim} 

\begin{proof}
If $xy$ crosses some edge on $\border({\cal R})$, then 
$xy$ must cross  either one  edge at least twice or at least two edges on 
$\border({\cal R})$, a contradiction. 
Thus, $xy$ does not cross any edge 
in $E(G_{i})$. 
By Claim~\ref{cl5.2.1}, all edges in $E_x\cup E_y\cup \{xy\}$ must lie in the interior of ${\cal R}$, with $v_1$ and $v_2$ on $\border({\cal R})$. 
It is clear that $xy$ does not cross any edge in $E_x\cup E_y$. 

Since $|\Nest_{i+1}^*|\ge 2$,
there must be crossings involving edges in $E_x\cup E_y$.
By Claim~\ref{cl5.2.1} and the fact
that $xy$ does not cross any edge in $E_x\cup E_y$, 
 such crossings must be between edges in  $E_x$ and 
edges in $E_y$.
If there is only one such crossing
$\alpha_1$ between 
edges $xa$ and $ya'$, 
then 
$\Nest^*_{i+1}$ has at most one 
nest ${\cal  N}= \langle \{a, a'\}, \{\alpha_1, \beta_1 \}\rangle$, where $\beta_1$ represents a crossing between two edges of $G_{i}$ incident with $a$ and $a'$,  contradicting the assumption. 
Thus, there are at least two crossings 
between edges in $E_x$ and edges in $E_y$.

Let $H$ denote the 1-plane graph 
$G_{i+1}-y$
inheriting the drawing of $G_{i+1}$. 
Then, ${\cal R}$ is divided into some
sub-faces. 
Assume that $y$ is within 
	a sub-face $R'$ of $R$, as shown in Figure~\ref{weak-aa}.
We may assume that ${\cal R}'$ is bounded by $e=xa, f=xb$ 
in $E_x$ and a segment $L$ of $\border({\cal R})$.
Clearly, $e$ and $f$ are the only edges 
in $E_x$ which can be crossed 
by edges in $E_y$.
Since there are at least two crossings 
between edges in $E_x$ and edges in $E_y$, 
there must be exactly two such 
crossings 
$\alpha_1$  and $\alpha_2$, where
$\alpha_1$ is the crossing between  $e=xa$ and  $e'=ya'$, 
and 
$\alpha_2$  is the crossing 
between $f=xb$ and  $f'=yb'$, 
where  $a'$ and $b'$ are on $\border({\cal R})$,
as illustrated in Figure \ref{weak-aa}.

\begin{figure}[!h]
	\centering

	\begin{tikzpicture}[scale=1.0,
		whitenode/.style={circle, draw=black, fill=white, minimum size=5mm, inner sep=1pt},
		pickshing/.style={draw=none, fill=blue!10}, shadow_blue/.style={fill=gray!20, draw=none}]
		\tikzset{t_edge/.style={}}
		\node [style=whitenode] (1) at (0.5635, -1.919) {};
		\node [style=whitenode] (2) at (1.5115, -1.3097) {$b'$};
		\node [style=whitenode] (4) at (1.8193, 0.8308) {$b$};
		\node [style=whitenode] (5) at (1.0813, 1.6825) {};
		\node [style=whitenode] (7) at (-1.0813, 1.6825) {$a$};
		\node [style=whitenode] (8) at (-1.8193, 0.8308) {};
		\node [style=whitenode] (9) at (-1.9796, -0.2846) {$a'$};
		\node [style=whitenode] (11) at (-0.5635, -1.919) {};
		\node [style=whitenode] (12) at (0.25, 0.75) {$y$};
		\node [style=whitenode] (13) at (-0.25, -0.5) {$x$};
		\node [style=none] (14) at (-2.25, -1.25) {};
		\node [style=none] (15) at (-1.5, -2) {};
		\node [style=none] (16) at (2.5, 0.5) {};
		\node [style=none] (17) at (2.25, -0.75) {};
		\node [style=none] (18) at (0, 2.25) {};
		\node [style=none] (19) at (-0.5, 2.5) {};
		\node [style=none] (20) at (0.5, 2.5) {};
		\node [style=none] (21) at (1.175, 0) {$\alpha_2$};
		\node [style=none] (22) at (1.55, 0.375) {$f$};
		\node [style=none] (23) at (0.5, -0.475) {$f'$};
		\node [style=none] (24) at (-1.35, 0.25) {$e'$};
		\node [style=none] (25) at (-0.9, 0.55) {$\alpha_1$};
		\node [style=none] (26) at (-1, 1) {$e$};
		\node [style=none] (27) at (-0.75, -1) {${\cal R}$};
		\node [style=none] (28) at (0.75, 2.25) {$L$};
		\node [style=none] (29) at (0.25, 1.5) {${\cal R}'$};
		\node [style=none] (30) at (-1.65, -1.475) {};
		\node [style=none] (31) at (2.1, -0.25) {};

		\draw[t_edge] (12) to (7);
		\draw[t_edge] (12) to (4);
		\draw[t_edge] (12) to (2);
		\draw[t_edge] (12) to (9);
		\draw[t_edge] (13) to (7);
		\draw[t_edge] (13) to (4);
		\draw[t_edge] (13) to (9);
		\draw[t_edge] (13) to (2);
		\draw[t_edge] (7) to (8);
		\draw[t_edge] (8) to (9);
		\draw[t_edge] (13) to (11);
		\draw[t_edge] (9) to (15.center);
		\draw[t_edge] (14.center) to (11);
		\draw[t_edge] (11) to (1);
		\draw[t_edge] (1) to (2);
		\draw[t_edge] (4) to (17.center);
		\draw[t_edge] (16.center) to (2);
		\draw[t_edge] (4) to (5);
		\draw[t_edge] (7) to (20.center);
		\draw[t_edge] (5) to (19.center);
		
		\draw[t_edge]  (7) to (-2.7, 0.7); 
		\draw[t_edge]  (9) to (-2.7, 0.9); 
		
	\node [style=none] () at (-2.7, 0.3) {$\beta_1$};
	\node [style=none] () at (2.5, -0.2) {$\beta_2$};

		\scoped[on background layer] {
			\draw [style={shadow_blue}] (4.center) to (13.center) to (7.center) to (18.center) to (5.center) to cycle; \draw [style=pickshing] (1.center) to (2.center) to (31.center) to (4.center) to (13.center) to (7.center) to (8.center) to (9.center) to (30.center) to (11.center) to cycle
			;}
	\end{tikzpicture}
	
	\caption{Both  $x$ and $y$ lie in a same face  of $G_i$}
	\label{weak-aa}
\end{figure}
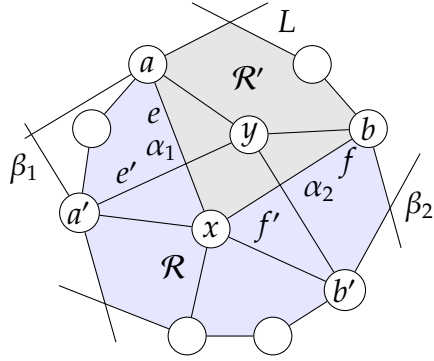
   
   Note that ${\cal  N}_0=\langle \{x, y\},  \{\alpha_1, \alpha_2 \}\rangle$ is a 
   nest in $\Nest^*_{i+1}$, 
    and 
   ${\cal  N}_1=\langle \{a, a'\}, \{\alpha_1, \beta_1\}\rangle$, 
   and ${\cal  N}_2=\langle \{b, b'\}, \{\alpha_2, \beta_2\}\rangle$
   are other possible nests in $\Nest^*_{i+1}$, 
      where $\beta_1$ is  a crossing between two edges of $G_{i}$ incident with $a$ and $a'$ respectively
      and $\beta_2$ is  a crossing between two edges of $G_{i}$ incident with $b$ and $b'$ respectively.
      Thus, $|\Nest_{i+1}^*|\le 3$. 
      
      Since $xa$ and $ya'$ cross each other, 
      $G[\{x,y,a,a'\}]\cong K_4$ 
      by Lemma~\ref{K_4}. 
      Similarly, $G[\{x,y,b,b'\}]\cong K_4$.
      Thus, $\{xa,xa',xb,xb'\}\subseteq E_x$
      and  $\{ya,ya',yb,yb'\}\subseteq E_y$.
      Now, it is known that $a\ne b$,  $a\ne a'$ and $b\ne b'$. 
      Since $xa$ and $xb$ are different 
      edges, and 
      $ya'$ crosses $xa$
      and $yb'$ crosses $xb$,
      we have $a'\ne b'$. 
      Obviously, $a\ne b'$, as
      $b'$ is not on $L$ or an end of $L$. 
      Similarly, $a'\ne b$. 
      Thus, $a,b,a'$ and $b'$ 
      are pairwise distinct vertices, 
      implying that 
      $|E_x|\ge 4$ and $|E_y|\ge 4$. 
      Claim~\ref{cl5.2.3} holds. 
\end{proof}

   By Claims~\ref{cl5.2.0}, 
   \ref{cl5.2.2} and \ref{cl5.2.3}, 
   inequality (\ref{le5.2-e1}) holds. 
   \end{proof}

\begin{lemma}
	\label{micro}
For any 
$i\in\brk{0,N-1}$, 
	if $F_{i+1}$ is a micro 
	$K_4$-link from $G_i$, 
then $\Nest_{i+1}^*=\emptyset$,
and 
\equ{micro-e}
{e(G_{i+1})-e(G_i)\geq 7,
}
where the equality holds if and only if 
$F_{i+1}$ is a simple micro $K_4$-link from $G_i$.
\end{lemma}

\begin{proof}
By the definition of  
micro $K_4$-link
and Lemma~\ref{micro-aa}, 
we have 
$e(G_{i+1})-e(G_i)\geq |E(F_{i+1})|+1=7$,
where the equality holds if and only if 
$F_{i+1}$ is a simple micro $K_4$-link from $G_i$.
In the following, we need only to show that 
$\Nest^*_{i+1}=\emptyset$. 

Let $V(F_{i+1})=\{x, y, z, u\}$, where $u\in V(G_i)$ and 
$x, y, z\in V(G_{i+1})\setminus V(G_i)$.
Let  
$E_x=\{xt\in E(G_{i+1}): t\in V(G_i)\}$, 
$E_y=\{yt\in E(G_{i+1}): t\in V(G_i)\}$
and 
$E_z=\{zt\in E(G_{i+1}): t\in V(G_i)\}$.
 
We are now going to prove the following claims. 
 
\setcounter{claim}{0}
\begin{claim}\label{cl-mi} 
	The following statements hold:
	
\vspace{-3 mm}
	
  \begin{itemize}[itemsep=-1mm]
 \item[(1)] no edge in $E_x\cup E_y\cup E_z\cup \{xy, yz, zx\}$ 
 crosses any edge in $E(G_{i})$;
 \item[(2)] any two edges in $E_x\cup E_y\cup E_z$ do not cross each other;
 and 
 \item[(3)] any two edges in $\{xy, yz, zx\}$ do not cross each other.
  \end{itemize}
\end{claim}

\begin{proof}  If (1) fails, 
	by Lemma~\ref{K_4}, 
	 there is either a strong  or  weak $K_4$-link from $G_i$, contradicting
	 the assumption  that $F_{i+1}$ is a micro $K_4$-link from $G_i$
	 chosen by the SWM*-rule. 
	 Thus, (1) holds.

Now prove (2). 
If edges $e_1$ and $e_2$ in 
$E_x\cup E_y\cup E_z$ cross each 
other, then they belongs to  
two distinct sets.
Then, by Lemma~\ref{K_4}, 
$G_{i+1}$ has a weak $K_4$-extension of $G_i$, a contradiction.
(3) is also clear, because any two edges in $\{xy, yz, zx\}$ share a common vertex.
\end{proof}

A crossing in $G_{i+1}$ is said to be new if it is not between two edges 
	in $G_i$.
By  (2) and (3) in Claim~\ref{cl-mi}, 
the following claim follows.

\begin{claim}\label{cl-mi-0} 
Any new crossing in $G_{i+1}$ 
must be between $xy$ and some edge in 
$E_z$, 
or $xz$ and some edge in $E_y$,
or $yz$ and some edge in $E_x$. 
\end{claim} 

\begin{claim}\label{cl-mi-1}
$G_{i+1}$ has at most one 
crossing between two edges in 
$\partial_{G_{i+1}}(x)
\cup \partial_{G_{i+1}}(y)
\cup \partial_{G_{i+1}}(z)$.
\end{claim} 

\begin{proof}
By (1) in Claim~\ref{cl-mi}, 
vertices $x$, $y$ and $z$, along with all the edges in 
$\partial_{G_{i+1}}(x)
\cup \partial_{G_{i+1}}(y)
\cup \partial_{G_{i+1}}(z)$
must lie within some face ${\cal R}$ of $G_{i}$. Obviously,  $u$ is on 
$\border({\cal R})$. 

Since $F_{i+1}\cong K_4$, 
there is at most one crossing 
between two 
edges in $F_{i+1}$, 
say $zu$ and $xy$, 
as shown in Figure~\ref{micro-f1} (b). 

\begin{figure}[!h]
	\begin{minipage}[b]{0.48\textwidth}
		\centering
		\begin{tikzpicture}
			\tikzset{vertex/.style={circle,draw=black,fill=white,inner sep=2pt},
				blackedge/.style={black,line width=1.5pt}, blueedge/.style={blue,line width=1.5pt}, blueedge_normal/.style={blue}
			}
			\fill[gray!20] (-2.3,-2.3) rectangle (2.5,2.5);
			\draw[fill=white, draw=black] (0,0) circle (2);
			\node (G)  at (1.8,1.8)  {$G_{i}$};
			
			\node[vertex]  (z)  at (0,1)  {$z$};
			
		\node (R)  at (1, -0.8)  {${\cal R}$};

			\def\radius{1}
			
			\pgfmathsetmacro{\ux}{0 + \radius*cos(90)} 
			\pgfmathsetmacro{\uy}{1 + \radius*sin(90)}
			\pgfmathsetmacro{\xx}{0 + \radius*cos(210)} 
			\pgfmathsetmacro{\xy}{1 + \radius*sin(210)}
			\pgfmathsetmacro{\yx}{0+ \radius*cos(-30)} 
			\pgfmathsetmacro{\yy}{1 + \radius*sin(-30)}
			
			\node[vertex] (u) at (\ux,\uy) {$u$}; 
			\node[vertex] (x) at (\xx,\xy) {$x$}; 
			\node[vertex] (y) at (\yx,\yy) {$y$}; 
			
			\node[vertex] (x1)  at (145:2)   {};
			\node[vertex]  (x2)  at (165:2)  {};
			\node[vertex]  (x3)  at (185:2)   {};
			\node[vertex] (y1)  at (30:2)   {};
			\node[vertex]  (y2)  at (10:2)  {};
			\node[vertex]  (y3)  at (-10:2)   {};
			\node[vertex]  (z1)  at (-90:2)   {$z'$};
			\node[]   at (166:1.7)   {$\vdots$};
			\draw[blueedge] (z) -- (u) -- (x) -- (z); 
			\draw[blueedge] (z) -- (y) -- (u); 
			\draw[blueedge] (x) -- (y); 
			\draw[orange, dashed] (z)--(z1); 
			\foreach \i in {1,2,3} {
				\draw[blueedge_normal] (x) -- (x\i);
				\draw[blueedge_normal] (y) -- (y\i);
			}
			\node[]  at (10:1.7)   {$\vdots$};
		\end{tikzpicture}
	\end{minipage}
	\hfill
	\begin{minipage}[b]{0.48\textwidth}
		\centering
		\begin{tikzpicture} 
			\tikzset{vertex/.style={circle,draw=black,fill=white,inner sep=2pt},
				blackedge/.style={black,line width=1.5pt}, blueedge/.style={blue,line width=1.5pt}, blueedge_normal/.style={blue}
			}
			\fill[gray!20] (-2.3,-2.3) rectangle (2.5,2.5);
			\draw[fill=white, draw=black] (0,0) circle (2);
			\node (G)  at (1.8,1.8)  {$G_{i}$};
			\node[vertex]  (z)  at (0,0)  {$z$};
			
			\node (R)  at (1, -0.3)  {${\cal R}$};
			
			\def\radius{0.5}
			
			\pgfmathsetmacro{\ux}{0 + \radius*cos(90)} 
			\pgfmathsetmacro{\uy}{1.5 + \radius*sin(90)}
			\pgfmathsetmacro{\xx}{-0.3 + \radius*cos(210)} 
			\pgfmathsetmacro{\xy}{1.2 + \radius*sin(210)}
			\pgfmathsetmacro{\yx}{0.3+ \radius*cos(-30)} 
			\pgfmathsetmacro{\yy}{1.2 + \radius*sin(-30)}
			
			\node[vertex] (u) at (\ux,\uy) {$u$}; 
			\node[vertex] (x) at (\xx,\xy) {$x$}; 
			\node[vertex] (y) at (\yx,\yy) {$y$}; 
			
			\node[vertex]  (x1)  at (150:2)   {};
			\node[vertex]  (x2)  at (170:2)  {};
			\node[vertex]  (x3)  at (190:2)   {};
			\node[vertex] (y1)  at (-10:2)   {};
			\node[vertex]  (y2)  at (10:2)  {};
			\node[vertex]  (y3)  at (30:2)   {};
			
			\node[vertex]  (z1)  at (-90:2)   {};
			\node[vertex]  (z2)  at (-120:2)  {};
			\node[vertex]  (z3)  at (-60:2)   {};
			\draw[blueedge] (z) -- (u) -- (x) -- (z); 
			\draw[blueedge] (z) -- (y) -- (u); 
			\draw[blueedge] (x) -- (y); 
			\foreach \i in {1,2,3} {
				\draw[blueedge_normal] (x) -- (x\i); 
				\draw[blueedge_normal] (y) -- (y\i);
				\draw[blueedge_normal] (z) -- (z\i);
			}
			\node[]   at (170:1.8)   {$\vdots$};
			\node[]   at (-75:1.8)   {$\dots$};
			\node[]   at (10:1.8)   {$\vdots$};
		\end{tikzpicture}
	\end{minipage}

\centerline{(a) \hspace{7.5 cm} (b)}

\caption{At most one 
crossing between 
edges in 
$\partial_{G_{i+1}}(x)
\cup \partial_{G_{i+1}}(y)
\cup \partial_{G_{i+1}}(z)$
} 
\label{micro-f1} 
\end{figure}
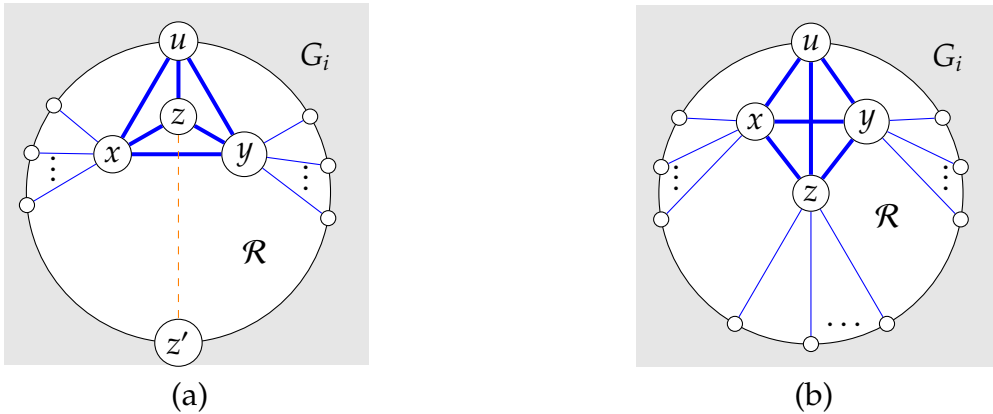

By Claim~\ref{cl-mi-0}, 
we need only to consider possible 
crossings between edges in 
$\{xy,xz,yz\}$ and edges 
in $E_x\cup E_y\cup E_z$.  

Obviously, in Figure~\ref{micro-f1} (b), 
no edge $zz'$ in $\partial_{G_{i+1}}(z)$ crosses $xy$,  where $z'\ne u$.
But, in Figure~\ref{micro-f1} (a), it is possible that 
some edge $zz'\in \partial_{G_{i+1}}(z)$ crosses $xy$, where $z'$ is on $\border({\cal R})$.
Also notice that in 
both (a) and (b) of Figure~\ref{micro-f1}, 
no edge  $xx'$ in $\partial_{G_{i+1}}(x)$ crosses $yz$, 
and no edge 
$yy'$ in $\partial_{G_{i+1}}(y)$ crosses $xz$, 
where $x'$ and $y'$ is on $\border({\cal R})$.
Hence Claim~\ref{cl-mi-1} holds.
\end{proof} 

Now we are going to complete the proof.  Suppose that 
	$\setn=\langle \{s,t\},  \{\alpha_1, \alpha_2 \}\rangle
	\in \Nest^*_{i+1}$, where 
	$s,t\in V(G_{i+1})$ and 
	$\alpha_1$ and $\alpha_2$ are crossing points. 
	If $\{s,t\}\cap \{x,y,z\}=\emptyset$, 
	then, either some edge in 
	$E_x\cup E_y\cup E_z$ crosses 
	edges in $G_{i}$, 
	or two edges in $E_x\cup E_y\cup E_z$
	cross each other,
	a contradiction to (1) and (2) of Claim~\ref{cl-mi}.
	
	Hence $\Nest^*_{i+1}=\emptyset$
	and the lemma holds.
	\end{proof}

\subsection{Proof of inequality~(\ref{e5-1})} 

In this subsection, we will prove 
inequality~(\ref{e5-1})  which is 
crucial 
for proving Theorem~\ref{main1.3},
the main result in this article.

\begin{proposition}\label{skeleton}
Let $G$ be a maximal $1$-plane graph 
in which each edge in $G$ is contained 
in some $K_4$-subgraph. 
If $n(G)\ge 5$,   then 
\equ{ske-e1}
{
e(G)\geq \frac{7}{3}n(G)+\frac{1}{3}\nest(G)-3. 
}
\end{proposition}

\begin{proof} 
By Proposition~\ref{exten0}, 
$G$ has a $K_4$-extension 
sequence 
$\{G_i\}_{i\in \brk{0,N}}$.
Thus, $G_0\cong K_4$
and $G_N=G$, 
and for each $i\in \brk{N}$, 
$G_i=G[V(G_{i-1})\cup F_i]$, 
where $F_i$ is a $K_4$-link from $G_{i-1}$ 
determined by the SWM*-rule. 
We assume that the selection of this $K_4$-extension sequence $\{G_i\}_{i\in \brk{0,N}}$ is subject to the following conditions:
\begin{enumerate}[itemsep=-1mm]
	\item[(i)] if $G$ has two different  $K_4$-subgraphs 
	which share three vertices, 
	then $F_1$ is a strong $K_4$-link 
	from $G_0$\footnote{
		If $G$ has such two $K_4$-subgraphs, then $G_0$ can be one of them, 
		and thus $F_1$ is a strong $K_4$-link from $G_0$.
},
	and 
	\item[(ii)] for any $i\in \brk{N}$, 
	if $F_i$ is a weak $K_4$-link and 
	at least one edge in $G_{i-1}$ is not contained in any $K_4$-subgraph 
	of $G_{i-1}$, 
	then the only edge in 
	$E(F_i)\cap E(G_{i-1})$ 
	 is not contained in any $K_4$-subgraph of $G_{i-1}$.
\end{enumerate}

We can first prove the following claim.

\setcounter{claim}{0}

\begin{claim}\label{sk-cl1} 
For each $j\in\brk{0,N}$,
\equ{sk-e1}
{
e(G_j)\geq \frac{7}{3}n(G_j)+\frac{1}{3}\nest(G_j)-\frac{10}3,
}
and 
if (\ref{sk-e1}) is strict  
$j=i$, where $i\in \brk{N}$, 
then  $e(G_s)\ge  \frac{7}{3}n(G_s)+\frac{1}{3}\nest(G_s)-3$ for all $s\in \brk{i,N}$.
\end{claim} 

\begin{proof}
	Since $G_0\cong K_4$ and has at most one crossing, 
	$\nest(G_0)=0$ and 
	(\ref{sk-e1}) holds for $j=0$. 
	By Lemmas~\ref{strong}, \ref{weak} and \ref{micro}, 
		it can be verified easily that 
		whenever (\ref{sk-e1}) holds for $j=i$, where $i\in\brk{0,N-1}$, 
		it also holds for $j=i+1$.
		If  (\ref{sk-e1}) is strict for some 
		$i\in \brk{N}$, then 
	$e(G_i)\ge  \frac{7}{3}n(G_i)
	+\frac{1}{3}\nest(G_i)-3$.
	By Lemmas~\ref{strong}, \ref{weak} and \ref{micro} again, 
	$e(G_s)\ge  \frac{7}{3}n(G_s)
+\frac{1}{3}\nest(G_s)-3$
for all $s\in \brk{i,N}$.	
\end{proof} 


Now we suppose that 
$G$ has the minimum order 
such that (\ref{ske-e1}) fails for $G$.
In the following, we establish some claims on $G$, which lead to a contradiction.

\begin{claim}\label{sk-cl2} 
For each $i\in\brk{0,N-1}$, 
$\Nest(G_{i})\subseteq \Nest(G_{i+1})$
and 
the equality of (\ref{le-st-e1})
or (\ref{le5.2-e1}) or (\ref{micro-e})
holds respectively when 
$F_{i+1}$ is a strong or weak
or micro  
$K_4$-link from $G_{i}$.
\end{claim} 

\begin{proof}
 By the assumption 
of $G$, 
Claim~\ref{sk-cl1} implies that 
	the equality of  (\ref{sk-e1}) holds for each $j\in\brk{0,N}$.

	Suppose that
	$\Nest(G_{i})\not \subseteq \Nest(G_{i+1})$
	for some $i\in \brk{0,N-1}$, 
	implying that 
	 	$\omega(G_{i+1})\le \omega(G_{i})
	 +|\Nest^*_{i+1}|-1$. By Lemmas~\ref{strong}, ~\ref{weak}
	 	and~\ref{micro}, 
	 	we have $e(G_{i+1})-e(G_i)
	 	\ge \frac {7}3 (n(G_{i+1})-n(G_i))
	 	+\frac{1}3 |\Nest^*_{i+1}|$.
	 	It follows that 
\eqn{sk-e2}
{
e(G_{i+1})
&\ge &e(G_i)+
\frac {7}3 (n(G_{i+1})-n(G_i))
+\frac{1}3 |\Nest^*_{i+1}|
\nonumber \\
&\ge &e(G_i)+
\frac {7}3 (n(G_{i+1})-n(G_i))
+\frac{1}3 (\omega(G_{i+1})-\omega(G_i)+1)
\nonumber \\
&=&e(G_i)-\frac 73 n(G_i)-\frac 13 \omega(G_i)
+\frac {7}3 n(G_{i+1})
+\frac{1}3\omega(G_{i+1})+\frac 13
\nonumber \\
&=&\frac {7}3 n(G_{i+1})
+\frac{1}3\omega(G_{i+1})-3,
}	 	
where the last equality follows from 
the equality of 
(\ref{sk-e1}) for $j=i$.
Thus, the inequality of 
(\ref{sk-e1}) is strict for $j=i+1$,
contradicting the fact that 
 the equality of  (\ref{sk-e1}) 
 holds for each $j$ with
$0\le j\le N$.

Similarly, if for some 
$i\in \brk{0,N-1}$,
the inequality of (\ref{le-st-e1})
	 or (\ref{le5.2-e1}) or (\ref{micro-e})
	 is strict respectively when 
	 $F_{i+1}$ is a strong or weak
	 or  micro  
	 $K_4$-link from $G_{i}$,
	 then,
	 	 $e(G_{i})\ge \frac{7}{3}n(G_{i})+\frac{1}{3}\nest(G_{i})-\frac{10}3$
	 	 implies that 
	 $e(G_{i+1})\ge \frac{7}{3}n(G_{i+1})+\frac{1}{3}\nest(G_{i+1})-3$,
	 contradicting the fact that 
	 the equality of  (\ref{sk-e1}) 
	 holds for each $j\in \brk{0,N}$.
	 
	Hence Claim~\ref{sk-cl2} holds.
\end{proof}

For any $i\in \brk{0,N-1}$, 
a weak $K_4$-link $F_{i+1}$ 
from $G_{i}$  is said to be 
{\it simple} 
if $e(G_{i+1})-e(G_i)=5$ 
and $|\Nest^*_{i+1}|=1$.

\begin{claim}\label{sk-cl3} 
	For each $i\in\brk{0,N-1}$, 
	$F_{i+1}$ is a strong or simple weak
	or simple micro  
	$K_4$-link from $G_{i}$.
\end{claim} 

\begin{proof}
	For each $i\in\brk{0,N-1}$, 
	by Claim~\ref{sk-cl2}, 
	the equality of (\ref{le-st-e1})
	or (\ref{le5.2-e1}) or (\ref{micro-e})
	holds respectively
	when 
	$F_{i+1}$ is a strong or weak
	or micro  
	$K_4$-link from $G_{i}$.
	The claim then follows directly from 
	Lemmas~\ref{weak} and~\ref{micro}.  
\end{proof}

\begin{claim}\label{sk-cl4} 
	Any two  different $K_4$-subgraphs
	in $G$ have at most two 
	vertices in common. 
\end{claim} 

\begin{proof}
Suppose that there are two 
different $K_4$-subgraphs in $G$ 
which have three  
vertices in common. 
By the assumption on $G$, 
$F_1$ is a strong $K_4$-link from $G_0$.

Clearly, $|e(G_1)|-e(G_0)|\ge 3$. 
By Claim~\ref{sk-cl2}, $e(G_1)-e(G_0)=
\frac 73+\frac 13 |\Nest^*_1|$,
implying that $|\Nest^*_1|\ge 2$.
Thus, $G_1$ has crossings.
It is not difficult to verify that 
$G_1$ is one of the $1$-plane graphs shown in Figure~\ref{skg-cl4}. 
It follows that $|\Nest^*_1|\le 1$,
a contradiction.  
Thus, Claim~\ref{sk-cl4} holds.
\end{proof}

	\begin{figure}[!h]
	\centering

\begin{tikzpicture}[scale=0.5]
	\tikzset{whitenode/.style={circle,draw=black,fill=white,inner sep=2pt},
			blackedge/.style={thick},blueedge/.style={blue,thick}
		}
	\begin{pgfonlayer}{nodelayer}
		\node [style=whitenode] (0) at (-10.5, 5) {};
		\node [style=whitenode] (1) at (-10.5, 3) {};
		\node [style=whitenode] (2) at (-10.5, 1) {};
		\node [style=whitenode] (3) at (-12, 0) {};
		\node [style=whitenode] (4) at (-9, 0) {};
		\node [style=none] (5) at (-10.5, -1) {(a)};
		\node [style=whitenode] (6) at (-5.5, 5) {};
		\node [style=whitenode] (7) at (-5.5, 3) {};
		\node [style=whitenode] (8) at (-5.5, 1) {};
		\node [style=whitenode] (9) at (-7, 0) {};
		\node [style=whitenode] (10) at (-4, 0) {};
		\node [style=none] (11) at (-5.5, -1) {(b)};
		\node [style=whitenode] (12) at (-2, 3) {};
		\node [style=whitenode] (13) at (1, 3) {};
		\node [style=whitenode] (14) at (-2, 0) {};
		\node [style=whitenode] (15) at (1, 0) {};
		\node [style=whitenode] (16) at (-0.5, 5) {};
		\node [style=none] (17) at (-0.5, -1) {(c)};
		\node [style=whitenode] (18) at (3, 3) {};
		\node [style=whitenode] (19) at (6, 3) {};
		\node [style=whitenode] (20) at (3, 0) {};
		\node [style=whitenode] (21) at (6, 0) {};
		\node [style=whitenode] (22) at (4.5, 5) {};
		\node [style=none] (23) at (4.5, -1) {(d)};
		\node [style=whitenode] (24) at (9, 3) {};
		\node [style=whitenode] (25) at (12, 3) {};
		\node [style=whitenode] (26) at (9, 0) {};
		\node [style=whitenode] (27) at (12, 0) {};
		\node [style=whitenode] (28) at (10.5, 2.5) {};
		\node [style=none] (29) at (10.75, -1) {(e)};
	\end{pgfonlayer}
	\begin{pgfonlayer}{edgelayer}
		\draw [style=blackedge] (2) to (3);
		\draw [style=blackedge] (2) to (4);
		\draw [style=blackedge] (4) to (3);
		\draw [style=blackedge] (3) to (1);
		\draw [style=blackedge] (1) to (4);
		\draw [style=blackedge] (1) to (2);
		\draw [style=blueedge] (0) to (1);
		\draw [style=blueedge, bend right=25, looseness=0.75] (0) to (2);
		\draw [style=blueedge] (0) to (4);
		\draw [style=blackedge] (8) to (9);
		\draw [style=blackedge] (8) to (10);
		\draw [style=blackedge] (10) to (9);
		\draw [style=blackedge] (9) to (7);
		\draw [style=blackedge] (7) to (10);
		\draw [style=blackedge] (7) to (8);
		\draw [style=blueedge] (6) to (7);
		\draw [style=blueedge, bend right=25, looseness=0.75] (6) to (8);
		\draw [style=blueedge, bend left=15] (6) to (10);
		\draw [style=blueedge, bend right=15] (6) to (9);
		\draw [style=blueedge] (16) to (12);
		\draw [style=blackedge] (12) to (13);
		\draw [style=blueedge] (13) to (16);
		\draw [style=blackedge] (12) to (14);
		\draw [style=blackedge] (14) to (15);
		\draw [style=blackedge] (15) to (13);
		\draw [style=blackedge] (13) to (14);
		\draw [style=blackedge] (12) to (15);
		\draw [style=blueedge, bend right=60, looseness=1.25] (16) to (14);
		\draw [style=blueedge] (22) to (18);
		\draw [style=blackedge] (18) to (19);
		\draw [style=blueedge] (19) to (22);
		\draw [style=blackedge] (18) to (20);
		\draw [style=blackedge] (20) to (21);
		\draw [style=blackedge] (21) to (19);
		\draw [style=blackedge] (19) to (20);
		\draw [style=blackedge] (18) to (21);
		\draw [style=blueedge, bend right=60, looseness=1.25] (22) to (20);
		\draw [style=blueedge, bend left=60, looseness=1.25] (22) to (21);
		\draw [style=blueedge] (28) to (24);
		\draw [style=blackedge] (24) to (25);
		\draw [style=blueedge] (25) to (28);
		\draw [style=blackedge] (24) to (26);
		\draw [style=blackedge] (26) to (27);
		\draw [style=blackedge] (27) to (25);
		\draw [style=blackedge] (25) to (26);
		\draw [style=blackedge] (24) to (27);
		\draw [style=blueedge, in=135, out=120, looseness=2.50] (28) to (26);
	\end{pgfonlayer}
\end{tikzpicture}

	
	
	\caption{Possible $1$-planar drawings for $G_1$
if $F_1$ is a strong $K_4$-link from $G_0$	
} 
	
	\label{skg-cl4} 
\end{figure}
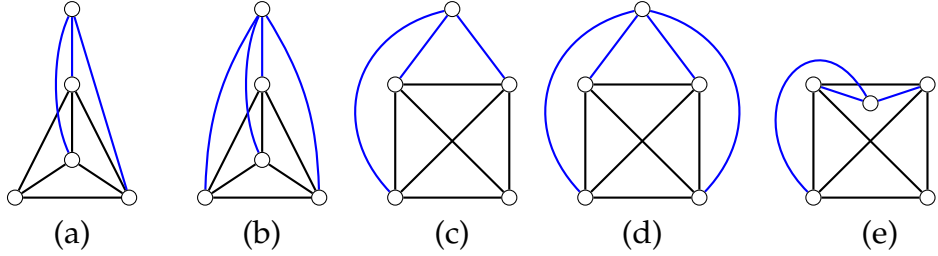 

\begin{claim}\label{sk-cl5} 
Let $i\in \brk{N}$. If $F_j$ is a weak $K_4$-link from $G_{j-1}$ for 
every $j\in \brk{i}$, then $G_i$ has the following properties:

\vspace{-3 mm}

\begin{enumerate}[itemsep=-1mm]
	\item[(i)] each edge in $G_i$ is contained in some $K_4$-subgraph 
	of $G_i$; 
	
	\item[(ii)] 
	each $K_3$-subgraph of $G_i$
	is contained in a $K_4$-subgraph of $G_i$;
	
	\item[(iii)] the external face of $G_i$ 
	is bounded by exactly $4+2i$ 
	clean edges, while each of its 
	internal faces is a false $3$-face; and 
	
	\item[(iv)] 
	there exists a clean edge within each nest in $\Nest_i$, 
	and every clean edge in $G_i$
	is either within some nest of $\Nest_i$ or on the  boundary of 
the external face of $G_i$.
\end{enumerate}  
\end{claim} 

\begin{proof}
Clearly, properties (i) and (ii) hold 
for $G_0$.
If $G_0$ has no crossing, then 
$\Nest^*_1=\emptyset$,
implying that 
$F_1$ is not a simple weak $K_4$-link 
from $G_0$, 
contradicting Claim~\ref{sk-cl3}.
Thus, $G_0$ has a crossing
and property (iii) also holds for $G_0$.

Now we assume that these three properties in the claim hold
for $G_{j}$, 
where $0\le j<i$. We will show that 
they also hold for $G_{j+1}$, 

	By Claim~\ref{sk-cl3} and the given condition, 
	$F_{j+1}:=G[\{x,y,u,v\}]$ is a simple weak $K_4$-link from $G_{j}$,
	where $uv$ is an edge in $G_{j}$
	and $x,y\in V(F_{j+1})\setminus V(G_{j})$.
	We claim that both $x$ and $y$ are in the external face $R$ of $G_j$.

By the assumption,  $G_j$ has properties (i)-(iv). 
Suppose $x$ is within some false $3$-face
$R_0$ of $G_j$. 
If the only clean edge $e_0$ 
of $R_0$ is not on $\border(R)$,
then by property (iv) for $G_j$, 
$x$ is within some nest of $\Nest_j$, 
implying that 
$\Nest_j\not\subseteq \Nest_{j+1}$,
contradicting Claim~\ref{sk-cl2}. 
Thus, $e_0$ is on $\border(R)$. 
Then, $y$ is within either $R_0$
or $R$, as shown in Figure~\ref{weak-K4-0}.
However, in this case, $\Nest^*_{j+1}=\emptyset$, 
contradicting the fact that $F_{j+1}$
is a simple weak $K_4$-link from $G_j$.

	 	\begin{figure}[!h]
	 	\centering
\tikzset{
  vertex/.style={circle, draw=black, fill=white, inner sep=2pt},  
  bluenode/.style={circle, draw=blue, fill=blue!20,minimum size=0.6em,inner sep=2pt},
  blackedge/.style={black, line width=1pt},                    
  blueedge/.style={blue, line width=0.8pt},                      
  blueedge_normal/.style={blue}                                  
}
\begin{tikzpicture}[scale=0.6,bezier bounding box]
 \draw[draw=none, fill=gray!20] (0,0) arc[start angle=40, end angle=-220, x radius=3.1cm, y radius=2.8cm];

    \node [vertex] (v) at (-0, 0) {$v$};
    \node [vertex] (u) at (-4.6, 0) {$u$};
    \node [vertex] (hu) at (-4.6, -3) {};
    \node [vertex] (hv) at (0, -3) {};
\draw[ blackedge] (u) to (hu);
\draw[ blackedge] (v) to (hv);
\draw[ blackedge] (hu) to (hv);
\draw[ blackedge] (u) to (hv);
\draw[ blackedge] (v) to (hu);
\draw[ blackedge] (v) to (u);

\node [bluenode] (x) at (-2.8, -0.8) {};
\node [bluenode] (y) at (-1.8, -0.8) { };
\node   at (-2.8, -0.8) { \tiny $x$};
\node  at (-1.8, -0.8) {\tiny $y$};
\node   at (-2.5, 0.3) {$e_0$};
\node   at (-2.3, -2.1) {$\alpha'$};
\draw[blueedge] (u) to (x);
\draw[blueedge] (u) to (y);
\draw[blueedge] (x) to (y);
\draw[blueedge] (x) to (v);
\draw[blueedge] (v) to (y);
\node  at (-2.5, -5.5) {(a)};
\end{tikzpicture}
\quad \quad \quad \quad \quad \quad \quad
\begin{tikzpicture}[scale=0.6,bezier bounding box]
 \draw[draw=none, fill=gray!20] (0,0) arc[start angle=40, end angle=-220, x radius=3.1cm, y radius=2.8cm];

    \node [vertex] (v) at (-0, 0) {$v$};
    \node [vertex] (u) at (-4.6, 0) {$u$};
    \node [vertex] (hu) at (-4.6, -3) {};
    \node [vertex] (hv) at (0, -3) {};
\draw[ blackedge] (u) to (hu);
\draw[ blackedge] (v) to (hv);
\draw[ blackedge] (hu) to (hv);
\draw[ blackedge] (u) to (hv);
\draw[ blackedge] (v) to (hu);
\draw[ blackedge] (v) to (u);

\node [bluenode] (x) at (-2.5, -0.8) {  $x$};
\node [bluenode] (y) at (-2.5, 1) { $y$};
\node   at (-2, 0.3) {$e_0$};
\node   at (-2.3, -2.1) {$\alpha'$};
\draw [blueedge] (u) to (x);
\draw [blueedge] (u) to (y);
\draw [blueedge] (x) to (y);
\draw [blueedge] (x) to (v);
\draw [blueedge] (v) to (y);
\node  at (-2.5, -5.5) {(b)};
\end{tikzpicture}
	
	
	\caption{$x$ is within a false $3$-face} 
	
	\label{weak-K4-0} 
\end{figure}
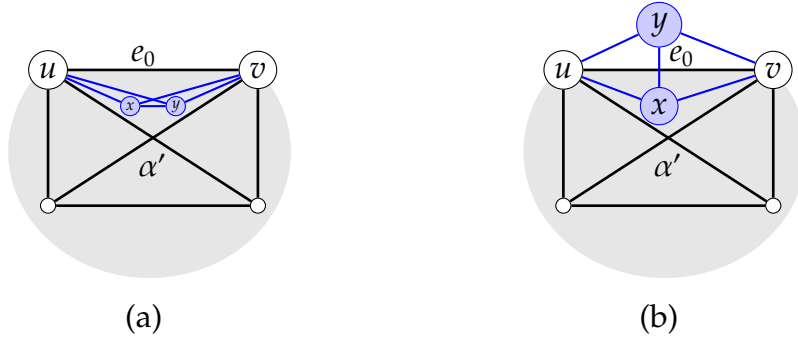 

	Thus,  both $x$ and $y$ are 
	within the external face $R$ of $G_j$.
By Claim~\ref{sk-cl2}, 
$G_{j+1}$ contains a new nest.
Such a new nest 
must be in the form 
$\langle \{u,v\},  \{\alpha, \alpha' \}\rangle$,
where $\alpha$ is a crossing 
between two edges in
$E_{G_{j+1}}(\{x,y\},\{u,v\})$ and 
$\alpha'$  is the crossing in $G_{j}$ 
such that 
$uv\alpha' u$ bounds a false $3$-face
in $G_{j}$, as shown in 
Figure~\ref{weak-K4} (b). 

Clearly,  $uv$ is a clean edge on $\border(R)$,
and for the external face $R'$ of $G_{j+1}$,
$\border(R')$ can be obtained from 
$\border(R)$ by replacing edge $uv$ by 
a path consisting of three clean edges
(e.g. path $uxyv$ in Figure~\ref{weak-K4} (b)).

	 	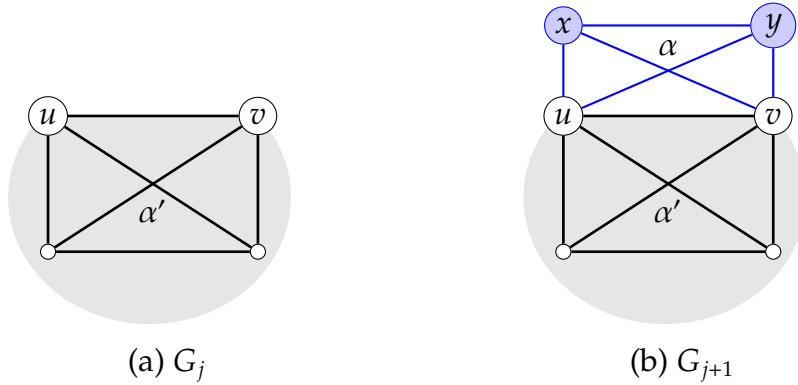
\begin{figure}[!h]
	 	\centering
	 	\tikzset{
  vertex/.style={circle, draw=black, fill=white, inner sep=2pt},  
  blackedge/.style={black, line width=1.0pt},                    
  blueedge/.style={blue, line width=0.8pt},                      
  blueedge_normal/.style={blue},
   bluenode/.style={circle, draw=blue, fill=blue!20,minimum size=0.6em,inner sep=2pt}                                 
}
\begin{tikzpicture}[scale=0.6,bezier bounding box]
 \draw[draw=none, fill=gray!20] (0,0) arc[start angle=40, end angle=-220, x radius=3.1cm, y radius=2.8cm];

    \node [vertex] (v) at (-0, 0) {$v$};
    \node [vertex] (u) at (-4.6, 0) {$u$};
    \node [vertex] (hu) at (-4.6, -3) {};
    \node [vertex] (hv) at (0, -3) {};
\draw[ blackedge] (u) to (hu);
\draw[ blackedge] (v) to (hv);
\draw[ blackedge] (hu) to (hv);
\draw[ blackedge] (u) to (hv);
\draw[ blackedge] (v) to (hu);
\draw[ blackedge] (v) to (u);

\node   at (-2.3, -2.1) {$\alpha'$};

\node  at (-2, -5.5) {(a) $G_j$};
\end{tikzpicture}
\quad \quad \quad \quad \quad \quad \quad
\begin{tikzpicture}[scale=0.6,bezier bounding box]
 \draw[draw=none, fill=gray!20] (0,0) arc[start angle=40, end angle=-220, x radius=3.1cm, y radius=2.8cm];

    \node [vertex] (v) at (-0, 0) {$v$};
    \node [vertex] (u) at (-4.6, 0) {$u$};
   \node [bluenode]  (x) at (-4.6, 2) {$x$};
    \node [bluenode] (y) at (0, 2) {$y$};
    \node [vertex] (hu) at (-4.6, -3) {};
    \node [vertex] (hv) at (0, -3) {};
\draw[ blackedge] (u) to (hu);
\draw[ blackedge] (v) to (hv);
\draw[ blackedge] (hu) to (hv);
\draw[ blackedge] (u) to (hv);
\draw[ blackedge] (v) to (hu);
\draw[ blackedge] (v) to (u);
\draw[blueedge] (x) to (u);
\draw[blueedge] (x) to (y);
\draw[blueedge] (y) to (v);
\draw[blueedge] (x) to (v);
\draw[blueedge] (y) to (u);
\node   at (-2.3, -2.1) {$\alpha'$};
\node   at (-2.3, 1.5) {$\alpha$};

\node  at (-2, -5.5) {(b) $G_{j+1}$};
\end{tikzpicture}
	 	
	 	
	 	\caption{$G_{j}$ and $G_{j+1}$} 
	 	
	 	\label{weak-K4} 
	 \end{figure} 
	 
Now it can be seen easily that properties 
(i), (ii), (iii) and (iv) hold for $G_{j+1}$.
The claim follows.
\end{proof}

\begin{claim}\label{sk-cl6} Assume that 
	$i\in \brk{N}$, and  
$F_j$ is a simple weak $K_4$-link
from $G_{j-1}$
for every $j\in \brk{i-1}$. 
Then
$F_{i}$ is not a simple micro $K_4$-link from $G_{i-1}$. 
\end{claim} 

\begin{proof}
Suppose that $F_{i}:=G[\{x,y,z,u\}]$
is a simple micro $K_4$-link from $G_{i-1}$, 
where $u$ is the only vertex of $F_{i}$ 
which is contained in $G_{i-1}$,
and $zv$ is the only edge in 
$E(G_{i})\setminus (E(F_{i})\cup E(G_{i-1}))$,
as shown in Figure~\ref{micro-ex} (a).

\begin{figure}[!h]
	\centering
\begin{tikzpicture}[scale=0.75]

		\tikzset{
			whitenode/.style={circle, draw=black, fill=white, minimum size=5mm,inner sep=2pt},
			whitenode2/.style={circle, draw=black, fill=white, minimum size=3mm},
			bluenode/.style={circle, draw=blue, fill=blue!20, minimum size=5mm},
			blackedge/.style={draw=black, line width=1pt},
			blueedge/.style={draw=blue},
			blueedge1/.style={draw=blue, line width=1.5pt} %
		}
		\draw[thick,fill=gray!20] (0, 3) ellipse (3.0cm and 1.2cm);
		\node   at (0, 3.5) {$G_{i-1}$};
		\node [whitenode] (u) at (-1, 2.5) {$u$};
		\node [whitenode] (v) at (1, 2.5) {$v$};
	    \node [whitenode] (x) at (-2.0, 0.5) {$x$};
        \node [whitenode] (z) at (0, 0.5) {$z$};
        \node [whitenode] (y) at (-1, -1.0) {$y$};
         \draw[ blackedge] (u) to (x) to (y) to (z) to (u);
         \draw[ blackedge] (v) to (z);
         \draw[ blackedge] (x) to (z);
         \draw[ blackedge] (u) to (y);

         \node   at (0, -2) {(a) $G_{i}$};
	\end{tikzpicture}
\quad \quad \quad \quad
\begin{tikzpicture}[scale=0.75]

		\tikzset{
			whitenode/.style={circle, draw=black, fill=white, minimum size=5mm,inner sep=2pt},
			whitenode2/.style={circle, draw=black, fill=white, minimum size=3mm},
			bluenode/.style={circle, draw=blue, fill=blue!20, minimum size=5mm},
			blackedge/.style={draw=black, line width=1pt},
			blueedge/.style={draw=blue,line width=0.5pt},
			blueedge1/.style={draw=blue, line width=1.5pt} %
		}
		\draw[thick,fill=gray!20] (0, 3) ellipse (3.0cm and 1.2cm);
		\node   at (0, 3.5) {$G_{i-1}$};
		\node [whitenode] (u) at (-1, 2.5) {$u$};
		\node [whitenode] (v) at (1, 2.5) {$v$};
	    \node [whitenode] (x) at (-2.0, 0.5) {$x$};
        \node [whitenode] (z) at (0, 0.5) {$z$};
        \node [whitenode] (y) at (-1, -1.0) {$y$};
        \node [whitenode] (z1) at (1, -1.0) {$z_1$};
		\node [whitenode] (z2) at (2, 0.5) {$z_2$};
         \draw[ blackedge] (u) to (x) to (y) to (z) to (u);
         \draw[ blackedge] (v) to (z);
         \draw[ blackedge] (x) to (z);
         \draw[ blackedge] (u) to (y);
         \draw[ blackedge] (z) to (z1);
          \draw[ blackedge] (z) to (z2);
         \draw[ blackedge] (z1) to (z2);
         \draw[ blackedge] (z2) to (z2);
         \draw[ blackedge] (v) to (z1);
          \draw[ blackedge] (v) to (z2);
         \node   at (0, -2) {(b) $G_{i+1}$};
	\end{tikzpicture}
	
\caption {$G_i$ is a simple micro $K_4$-extension of $G_{i-1}$} 
	
	\label{micro-ex} 
\end{figure}
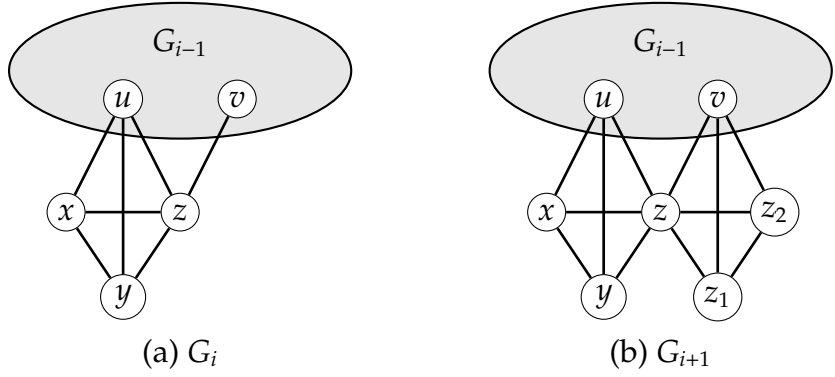 

By Claim~\ref{sk-cl5}, 
$zv$ is the only edge in $E(G_i)$
which is not contained 
in any $K_4$-subgraph of $G_i$,
and any $K_3$-subgraph $H$ 
of $G_i$ 
is contained in some $K_4$-subgraph 
of $G_i$, unless $zv$ is an edge in $H$.

Since $zv$ is not contained in any 
$K_4$-subgraph of $G_i$,
we have $i<N$. 
We are now 
going to show that 
$F_{i+1}$ is not a  strong
$K_4$-link from $G_i$.

Suppose that $F_{i+1}=\{x',w_1,w_2,w_3\}$ is  a strong $K_4$-link from $G_i$,
where $x'$ is the only vertex in 
$V(F_{i+1})\setminus V(G_i)$.
Since $F_i$ is a micro $K_4$-link 
from $G_{i-1}$,  
by the SWM*-rule, 
we have
$|\{w_1,w_2,w_3\}\cap V(G_{i-1})|\le 1$,
implying that 
$|\{w_1,w_2,w_3\}\cap \{x,y,z\}|\ge 2$.
Since $G$ does not have two different  
$K_4$-subgraphs which contain
three vertices in common, 
we have $u\notin \{w_1,w_2,w_3\}$
and  $|\{w_1,w_2,w_3\}\cap \{x,y,z\}|=2$.
It follows that $v \in \{w_1,w_2,w_3\}$,
implying that 
$|N_G(v)\cap \{x,y,z\}|\ge 2$, 
contradicting the fact that 
$F_i$ is a simple micro $K_4$-link from $G_{i-1}$. 
Hence $F_{i+1}$ is not a strong 
$K_4$-link from $G_i$.

By the given condition and Claim~\ref{sk-cl5}, each edge in $G_{i-1}$ is contained in some 
$K_4$-subgraph of $G_{i-1}$, 
implying that 
$zv$ is the only edge in $E(G_i)$
which is not contained 
in any $K_4$-subgraph of $G_i$.
Due to condition (ii)
for choosing the $K_4$-extension 
sequence $\{G_i\}_{0\le i\le N}$, 
mentioned before Claim~\ref{sk-cl1}, 
$F_{i+1}$ is a weak $K_4$-link 
from $G_i$ such that 
$zv\in E(F_{i+1})$. 
Assume that $F_{i+1}:=G[\{z,v,z_1,z_2\}$,
where $\{z_1,z_2\}= V(F_{i+1})\setminus V(G_i)$.

By Claim~\ref{sk-cl3}, $F_{i+1}$ is a 
simple weak $K_4$-link from $G_i$,
implying that $d_{G_{i+1}}(z_s)=3$ 
for both $s=1,2$,
and $|\Nest^*_{i+1}|=1$.
As $|\Nest^*_{i+1}|=1$, 
	there is a crossing 
between two edges in $F_{i+1}$,
and $zv$ is not crossed by $z_1z_2$;
otherwise, 
$|\Nest^*_{i+1}|=0$, a contradiction.
Thus, $zz_{s}$ crosses $vz_{3-s}$ for some $s\in [2]$,
say $s=2$, as shown in Figure~\ref{micro-ex} (b). 
As $zv$ is not on the boundary of 
any false $3$-face, 
no new nest is produced, 
implying that $|\Nest^*_{i+1}|=0$, a contradiction too.
 Hence Claim~\ref{sk-cl6} holds. 
\end{proof}

\begin{claim}\label{sk-cl7} 
	For each $i\in \brk{N}$, 
	$F_{i}$ is a simple weak
	$K_4$-link from $G_{i-1}$. 
\end{claim}

\begin{proof}
Suppose that the claim fails, 
and $s$ is the minimum number in $\brk{N}$
such that $F_{s}$ is not a 
simple weak
$K_4$-link from $G_{s-1}$.
If $s<N$, 
by Claim~\ref{sk-cl6}, 
$F_s$ is not a micro $K_4$-link from 
$G_{s-1}$. 
If $s=N$,  $F_N$ is not a simple micro $K_4$-link from 
$G_{N-1}$, otherwise, 
some edge in $G$ is not contained
in any $K_4$-subgraph, a contradiction.  
Thus, 
$F_{s}:=G[\{x,u,v,w\}]$ is a strong $K_4$-link from 
$G_{s-1}$, 
where $x$ is the only vertex in 
$V(F_s)\setminus V(G_{s-1})$.
By Claim~\ref{sk-cl5}, vertices 
$u,v$ and $w$  are contained 
in some $K_4$-subgraph of $G_{s-1}$,
implying that $G_s$ has two $K_4$-subgraphs sharing $3$ vertices,  
a contradiction to Claim~\ref{sk-cl4}.
Thus, 
Claim~\ref{sk-cl7} holds by Claim~\ref{sk-cl3}.
\end{proof}

Now we are going to complete the proof by Claims~\ref{sk-cl5} and~\ref{sk-cl7}.  
By  Claim~\ref{sk-cl7}, 
$F_{i}$ is a simple weak
$K_4$-link from $G_{i-1}$
for each $i\in \brk{N}$.
Then, by Claim~\ref{sk-cl5},
the external face $R$ of $G$ is bounded 
by exactly $4+2N$ clean edges,
implying that $R$ has $4+2N$ 
vertices in $G$. 
Lemma~\ref{2-con1} (a)
yields that $N=0$,
implying that $n(G)=4$, a contradiction. 

Hence the result holds.
\end{proof}

\section{Proof of Theorem~\ref{main1.3}}
\label{sec6}

In order to prove Theorem~\ref{main1.3},
we first provide a family of 
maximal $1$-plane 
graphs $H_n$ of order $n$
and size 
$\Ceil{\frac {7n}3}-3$
 for $n\ge 5$.

\subsection{Maximal $1$-plane graphs $H_n$ for $n\ge 5$} 
 
For $n=5,6,7$, $H_n$ is shown in 
Figure~\ref{order5-7}, which is obviously maximal.

\begin{figure}[!h]
	\centering

\begin{tikzpicture}[scale=0.8,bezier bounding box]
\tikzset{
			whitenode/.style={circle, draw=black, fill=white, minimum size=0.3em,inner sep=1.5pt},
			bluenode/.style={circle, draw=blue, fill=blue!20, minimum size=5mm},
			blackedge/.style={draw=black, line width=1pt},
			blueedge/.style={draw=blue},
			blueedge1/.style={draw=blue, line width=1.5pt} %
		}
	\begin{pgfonlayer}{nodelayer}
		\node [style=whitenode] (0) at (-7, 3) {$u$};
		\node [style=whitenode] (1) at (-4, 3) {$w_0$};
		\node [style=whitenode] (2) at (-7, 0) {};
		\node [style=whitenode] (3) at (-4, 0) {};
		\node [style=whitenode] (4) at (-5.5, 2.5) {$w_1$};
		\node [style=none] (5) at (-5.5, 0.75) {$\alpha_0$};
		\node [style=none] (6) at (-5.5, 3.5) {$\alpha_1$};
		\node [style=whitenode] (7) at (-1.5, 3) {$u$};
		\node [style=whitenode] (8) at (1.5, 3) {$w_0$};
		\node [style=whitenode] (9) at (-1.5, 0) {};
		\node [style=whitenode] (10) at (1.5, 0) {};
		\node [style=whitenode] (11) at (0, 2.5) {$w_1$};
		\node [style=none] (12) at (0, 0.75) {$\alpha_0$};
		\node [style=none] (13) at (0, 3.5) {$\alpha_1$};
		\node [style=whitenode] (14) at (-1, 1.5) {};
		\node [style=whitenode] (15) at (3.5, 3) {$u$};
		\node [style=whitenode] (16) at (6.5, 3) {$w_0$};
		\node [style=whitenode] (17) at (3.5, 0) {};
		\node [style=whitenode] (18) at (6.5, 0) {};
		\node [style=whitenode] (19) at (5, 2.5) {$w_1$};
		\node [style=none] (20) at (5, 0.75) {$\alpha_0$};
		\node [style=none] (21) at (5, 3.5) {$\alpha_1$};
		\node [style=whitenode] (22) at (4, 1.5) {};
		\node [style=whitenode] (23) at (8, 0) {};
		\node [style=none] (24) at (-5.75, -1) {(a) $H_5$};
		\node [style=none] (25) at (0, -1) {(b) $H_6$};
		\node [style=none] (26) at (5.25, -1) {(c) $H_7$};
	\end{pgfonlayer}
	\begin{pgfonlayer}{edgelayer}
		\draw [style=blackedge] (4) to (0);
		\draw [style=blackedge] (0) to (1);
		\draw [style=blackedge] (1) to (4);
		\draw [style=blackedge] (0) to (2);
		\draw [style=blackedge] (2) to (3);
		\draw [style=blackedge] (3) to (1);
		\draw [style=blackedge] (1) to (2);
		\draw [style=blackedge] (0) to (3);
		\draw [style=blackedge, in=135, out=120, looseness=2.50] (4) to (2);
		\draw [style=blackedge] (11) to (7);
		\draw [style=blackedge] (7) to (8);
		\draw [style=blackedge] (8) to (11);
		\draw [style=blackedge] (7) to (9);
		\draw [style=blackedge] (9) to (10);
		\draw [style=blackedge] (10) to (8);
		\draw [style=blackedge] (8) to (9);
		\draw [style=blackedge] (7) to (10);
		\draw [style=blackedge, in=135, out=120, looseness=2.50] (11) to (9);
		\draw [style=blackedge] (7) to (14);
		\draw [style=blackedge] (14) to (9);
		\draw [style=blackedge] (19) to (15);
		\draw [style=blackedge] (15) to (16);
		\draw [style=blackedge] (16) to (19);
		\draw [style=blackedge] (15) to (17);
		\draw [style=blackedge] (17) to (18);
		\draw [style=blackedge] (18) to (16);
		\draw [style=blackedge] (16) to (17);
		\draw [style=blackedge] (15) to (18);
		\draw [style=blackedge, in=135, out=120, looseness=2.50] (19) to (17);
		\draw [style=blackedge] (15) to (22);
		\draw [style=blackedge] (22) to (17);
		\draw [style=blackedge, bend right=15] (17) to (23);
		\draw [style=blackedge] (23) to (18);
		\draw [style=blackedge] (16) to (23);
	\end{pgfonlayer}
\end{tikzpicture}

	
%
	\caption{Maximal $1$-plane graphs 
		$H_n$ of order $n$ for $n=5,6,7$} 
	
	\label{order5-7} 
\end{figure}
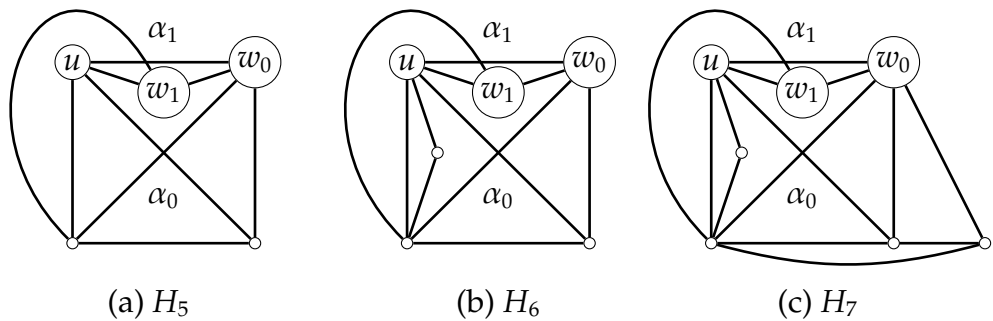

Observed from Figure~\ref{order5-7}, 
each of $H_5, H_6$ and $H_7$ contains a local structure $D$, 
as shown in 
Figure~\ref{insert} (a), 

	that can be described as a region
	bounded by $u\alpha_0w_0\alpha_2u$,
	containing 	a vertex $w_1$ 
	within 
	a false $3$-face
	of $H_i-w_1$ bounded by 
	$uw_0\alpha_0u$, 
	where $5\le i\le 7$
	and $\alpha_0$ is a crossing,
	such that $w_1$ is incident with 
	three edges, one of which 
	crosses 
	edge $uw_0$ at $\alpha_1$. 

Such a local structure 
$D$ can be written as $
\langle (u, w_0,w_1), (\alpha_0,  \alpha_1)\rangle$,
which does not represent a graph, but part of 
a $1$-plane graph. 
Let $\Omega(H)$ denote the family of 
such local structures in a $1$-plane graph $H$.  
Clearly, $\Omega(H_i)$ has exactly one member for each $i\in \brk{5,7}$.

\begin{figure}[H]
	\centering
\tikzset{
			whitenode/.style={circle, draw=black, fill=white, minimum size=0.3em,inner sep=0pt},
			bluenode/.style={circle, draw=blue, fill=blue!20, minimum size=5mm},
			blackedge/.style={draw=black, line width=1pt},
			blueedge/.style={draw=blue},
			blueedge1/.style={draw=blue, line width=1.5pt} %
		}
\begin{tikzpicture}[scale=0.6]
	\begin{pgfonlayer}{nodelayer}
		\node [style=none] (0) at (-6.5, 5) {};
		\node [style=none] (1) at (-4.5, 5) {};
		\node [style=whitenode] (2) at (-9, 0) {$u$};
		\node [style=whitenode] (3) at (-2, 0) {$w_0$};
		\node [style=whitenode] (4) at (-5.5, 1.5) {$w_1$};
		\node [style=none] (5) at (-5.5, -0.75) {};
		\node [style=none] (6) at (-5.5, 4.5) {$\alpha_0$};
		\node [style=none] (7) at (-6, -0.4) {$\alpha_1$};
		\node [style=none] (8) at (4.5, 5) {};
		\node [style=none] (9) at (6.5, 5) {};
		\node (arrow) at (0, 2) {\Large $\Rightarrow$};
		\node [style=whitenode] (10) at (2, 0) {$u$};
		\node [style=whitenode] (11) at (9, 0) {$w_0$};
		\node [style=whitenode] (12) at (5.5, 3) {$a$};
		\node [style=none] (14) at (5.5, 4.5) {$\alpha_0$};
		\node [style=none] (15) at (5, -0.375) {$\alpha_1$};
		\node [style=whitenode] (16) at (5.5, 1.5) {$w_1$};
		\node [style=whitenode] (17) at (6, 0.5) {$b$};
		\node [style=none] (18) at (5.5, -0.75) {};
		\node [style=whitenode] (19) at (4.5, 0.5) {$w_2$};
		\node [style=none] (20) at (5.25, 2) {};
		\node [style=none] (21) at (4.15, 1.325) {$\alpha_2$};
	\end{pgfonlayer}
	\begin{pgfonlayer}{edgelayer}
		\draw [style=blackedge] (0.center) to (3);
		\draw [style=blackedge] (1.center) to (2);
		\draw [style=blackedge] (2) to (3);
		\draw [style=blackedge] (4) to (5.center);
		\draw [style=blackedge] (4) to (2);
		\draw [style=blackedge] (4) to (3);
		\draw [style=blackedge] (8.center) to (11);
		\draw [style=blackedge] (9.center) to (10);
		\draw [style=blackedge] (12) to (10);
		\draw [style=blackedge] (12) to (11);
		\draw [style=blackedge] (16) to (11);
		\draw [style=blackedge] (16) to (17);
		\draw [style=blackedge] (17) to (11);
		\draw [style=blackedge] (16) to (18.center);
		\draw [style=blackedge] (10) to (11);
		\draw [style=blackedge] (16) to (19);
		\draw [style=blackedge] (19) to (10);
		\draw [style=blackedge] (10) to (16);
		\draw [style=blackedge] (11)
			 to [in=30, out=150, looseness=0.75] (20.center)
			 to [in=90, out=-150, looseness=0.75] (19);
	\end{pgfonlayer}
\end{tikzpicture}

	(a) $D=\langle \{u, w_0,w_1\}, \{\alpha_0,  \alpha_1\}\rangle$
	\hspace{3 cm}
	(b) $\phi(D)$ 
	
	\caption{$\phi(D)$ is obtained 
		from $D$ by inserting three new vertices and seven new edges} 
	
	\label{insert} 
\end{figure}
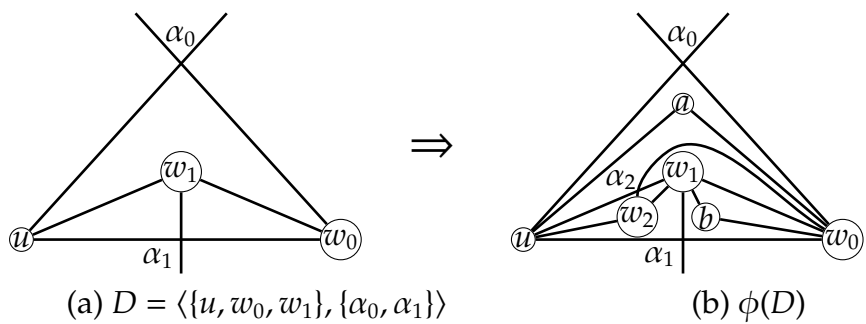

Let $H$ be a $1$-plane graph,  
and $D\in \Omega(H)$, 
as shown in Figure~\ref{insert} (a). 
Let $\phi(H)$ denote the $1$-plane 
graph obtained from $H$ by converting 
$D$ into $\phi(D)$,
as shown in Figure~\ref{insert} (b),
where $\phi(D)$ is obtained  
from $D$ 
by inserting $3$ 
new vertices $w_2$, $a$ and $b$
which are within faces of $H$ 
bounded by  $u\alpha_1w_1u$, $uw_1w_0\alpha_0w$
and $w_0w_1\alpha_1w_0$,
respectively,  
and $7$ new edges 
$w_2u, w_2w_1, w_2w_0,
au,aw_0, bw_0$ and 
$bw_1$, where $w_2w_0$ 
is the only crossing edge, 
which crosses edge $uw_1$
at point $\alpha_2$.
Observe that 
$\Omega(\phi(H))$ has a member 
$
\langle (u, w_1,w_2), (\alpha_1,  \alpha_2)\rangle$
within $\phi(D)$.
Thus, the $\phi(H)$ has the following property.

\begin{lemma}\label{le6-1}
	Let $H$ be a $1$-plane graph with
	$D\in \Omega(H)$. 
	If $\phi(H)$ is the graph introduced above, 
	then $\Omega(\phi(H))\ne \emptyset$, and whenever $H$ is maximal $1$-plane, $\phi(H)$ is too.
\end{lemma} 

\proof Assume that 
$D$ is $\langle (u, w_0,w_1), (\alpha_0,  \alpha_1)\rangle$, as shown in Figure~\ref{insert} (a). 
Then $\phi(D)$ is shown in 
Figure~\ref{insert} (b).
Clearly, $\Omega(\phi(H))$ has a member 
$
\langle (u, w_1,w_2), (\alpha_1,  \alpha_2)\rangle$
within $\phi(D)$. 
Thus, $\Omega(\phi(H))\ne \emptyset$.

Suppose that $H$ is a maximal $1$-plane graph.
Suppose that $\phi(H)$ is not 
maximal, and a new edge $e$ can be added to $\phi(H)$ to get a new 
$1$-plane graph.
By definition, 
$\phi(H)$ is obtained from $H$ by converting $D$ into $\phi(D)$.
It can be seen from Figure~\ref{insert} (b) that 
$e$ cannot be within 
the finite region of $\phi(H)$ 
bounded by
$uw_0\alpha_0u$, 
or crosses any edge on this boundary. 
Thus, $e$ must be within the infinite region of $\phi(H)$ bounded by
$uw_0\alpha_0u$,
implying that $e$ can be added to $H$ 
to get a larger $1$-plane graph, 
contradicting the assumption of $H$.

Hence the result holds.
\proofend

For any $n\ge 8$, let $H_n$ denote 
the $1$-plane graph $\phi(H_{n-3})$.  Then, we have a 
sequence 
of $1$-plane graphs 
$\{H_n\}_{n\ge 5}$, 
 which have the following property.

\begin{proposition}\label{Hn} 
For any $n\ge 5$, 
$\Omega(H_n)\ne \emptyset$ and 
$H_n$ is a maximal
$1$-plane graph of order $n$ 
and size $\Ceil{\frac {7n}3}-3$.
\end{proposition} 

\proof We will prove 
this proposition by induction on $n$. 
Clearly, the conclusions hold
for $n=5,6,7$ by Figure~\ref{order5-7}. 
Now assume that $n\ge 8$.
By the assumption, 
$H_n$ is the graph $\phi(H_{n-3})$.

Since $\Omega(H_{n-3})\ne \emptyset$ and 
$H_{n-3}$ is a maximal $1$-plane graph,
Lemma~\ref{le6-1} yields that 
$\Omega(H_{n})\ne \emptyset$ 
and $H_{n}$ is a maximal $1$-plane graph.
Observe that
$$
|E(H_n)|=|E(H_{n-3})|+7
=\Ceil{\frac{7(n-3)}{3}}-3+7
=\Ceil{\frac{7n}{3}}-3.
$$
Hence the conclusions hold 
for $H_n$. The proof is completed. 
\proofend 

\subsection{Proof of Theorem~\ref{main1.3}}

Let $G$ be any maximal 1-plane graph. 
Recall that 
$\hat{G}$ denotes the skeleton of $G$,
i.e., the subgraph obtained from $G$ 
by removing all hermits.   
We call an edge of $\hat{G}$
\emph{ exceptional} if it is not contained in any $K_4$-subgraph  in $\hat{G}$. 

We will apply the 
following important result, 
due to Bar\'{a}t and  T\'{o}th, which 
provides a structural property 
on the faces of 
$\hat G$ which contain an exceptional edge on its boundaries. 

\begin{lemma}[\cite{BT}]
	\label{skelet-3}
Let $\hat{G}$ be the skeleton of a maximal 1-plane graph $G$ with $n(\hat{G})\geq 4$.
If $ab$ is an exceptional 
edge in $\hat{G}$
and $F_1$ and $F_2$ are the 
two faces of 
$\hat{G}$ whose boundaries contain
$ab$, 
then

\vspace{-3 mm}

\begin{itemize}[itemsep=-1mm]
 \item[(i)] there exists a vertex $f$ 
 	in $\hat{G}$ 
 	such that $a,b$ and $f$ are the only vertices on both 
 $\border(F_1)$ and $\border(F_2)$; and 

 \item[(ii)] edges $af$ and $bf$  in 
 $\hat G$ are not exceptional, as shown in Figure~\ref{local-str}.
\end{itemize}
\end{lemma}


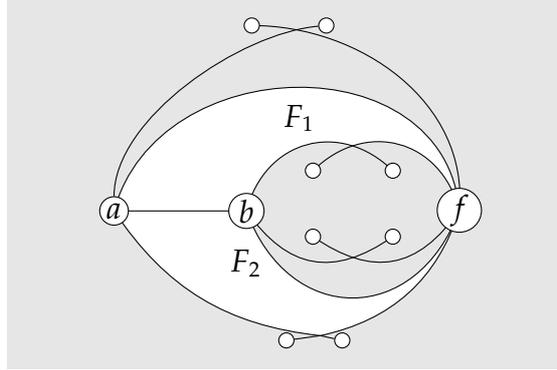
\begin{figure}[!h]
	\centering
	\begin{tikzpicture}[scale=0.35,bezier bounding box,
		whitenode/.style={circle, draw=black, fill=white, minimum size=2mm, inner sep=1pt},bluenode/.style={circle, draw=black, fill=white, minimum size=2mm, inner sep=1pt},shadow_Lightgrey2/.style={draw=none, fill=gray!20},shadow_white/.style={draw=none, fill=white}]
		\tikzset{blackedge/.style={}}
		\node [style=whitenode] (b) at (-4, -0.025) {$b$};
		\node [style=whitenode] (a) at (-8.975, -0.025) {$a$};
		\node [style=whitenode] (x1) at (-1.5, 1.5) {};
		\node [style=whitenode] (x2) at (1.5, 1.5) {};
		\node [style=whitenode] (y1) at (-1.5, -1) {};
		\node [style=whitenode] (y2) at (1.5, -1) {};
		\node [style=whitenode] (z1) at (-3.8, 7) {};
		\node [style=whitenode] (z2) at (-1, 7) {};
		\node [style=whitenode] (f) at (4.0, 0) {$f$};
		\node [style=whitenode] (w1) at (-2.5, -4.93) {};
		\node [style=whitenode] (w2) at (-0.4, -4.93) {};
		\node [style=none] (F2) at (-4, -2) {$F_2$};
		\node [style=none] (F1) at (-2, 3.5) {$F_1$};
		\node [style=none] (w3) at (-1.225, -4.75) {};
		
		\scoped[on background layer]{\fill[gray!20] (-13, -6) rectangle (8, 8);}  
		
		\scoped[on background layer]{\draw [style={shadow_white}] (f.center)
			to [in=75, out=95, looseness=1.25] (a.center)
			to [bend right=25] (w3.center)
			to [bend right] cycle;
			
			\draw [style={shadow_Lightgrey2}] (x1.center)
			to [bend left=60, looseness=1.25] (f.center)
			to [bend left=45, looseness=1.25] (y1.center)
			to (y2.center)
			to [bend right=315, looseness=1.25] (b.center)
			to [bend left=60, looseness=1.25] (x2.center)
			to cycle;	
			
			\draw [style={shadow_Lightgrey2}] (x1.center)
			to (x2.center)
			to [bend right=60, looseness=1.25] (b.center)
			to [in=250, out=-70, looseness=1.5] (f.center)
			to [in=45, out=105, looseness=1.25] cycle;
			\draw [style=blackedge,bend right=25] (a.center) to (w3.center);
			\draw [style=blackedge,in=250, out=-70, looseness=1.5] (b.center) to (f.center);}
		
		\draw [style=blackedge](w3.center) to (w2);
		
		\draw[style=blackedge] (w3.center) to (w1);
		\scoped[on background layer]{\draw [style=blackedge] (a.center) to (b.center);
			\draw [blackedge,bend left] (f.center) to (w3.center);
			\draw [blackedge, bend left=60, looseness=1.25] (b.center) to (x2.center);
			\draw [blackedge, bend right=45, looseness=1.25] (b.center) to (y2.center);
			\draw [blackedge, in=90, out=185, looseness=0.75] (z2) to (a);
			\draw [blackedge, in=85, out=0] (z1.center) to (f.center);
			\draw [blackedge, bend left=60, looseness=1.25] (x1.center) to (f.center);
			\draw [blackedge, bend left=45, looseness=1.25] (f.center) to (y1.center);
			\draw [blackedge, in=95, out=75, looseness=1.25] (a.center) to (f.center);}
		
	\end{tikzpicture}
	
	\caption{A local structural description of $G$ with an exceptional edge $ab$ }
	\label{local-str}
\end{figure}

We now proceed to prove our main result.

\vspace{0.2 cm} 

\noindent
{\bf Proof of Theorem~\ref{main1.3}}.  
By Proposition~\ref{Hn}, 
 it suffices to prove the following inequality  
\equ{main-e1}
{
e(G)\geq \frac{7}{3}n(G)-3,
}
for all maximal $1$-plane graphs $G$ with $n(G)\ge 5$. 

We will prove (\ref{main-e1}) by contradiction. 
Suppose that (\ref{main-e1}) fails, and $G$ is a maximal $1$-plane graph
with the minimum order such that it
is a 
counter-example to (\ref{main-e1}). 

\setcounter{claim}{0}

\begin{claim}\label{main-cl0} 
	$n(G)\ge 6$. 
\end{claim} 

\begin{proof}
	Suppose that $n(G)=5$. 
	The assumption of $G$ yields that 
	$e(G)< \frac 73 \times 5-3= \frac {26}3$,
	implying that $e(G)\le 8$. 
However, it is routine to verify that 
	there are only two possible 
	 maximal $1$-plane graphs 
	 of order $5$,
	as shown in Figure~\ref{K4}.
	Thus, $e(G)\ge 9$, a contradiction.  

	Claim~\ref{main-cl0} holds.
\end{proof}

\begin{figure}[!h]
\centering
\tikzset{
			whitenode/.style={circle, draw=black, fill=white, minimum size=0.5em, inner sep=1pt},
			bluenode/.style={circle, draw=blue, fill=blue!20, minimum size=5mm},
			blackedge/.style={draw=black, line width=1pt},
			blueedge/.style={draw=blue},
			blueedge1/.style={draw=blue, line width=1.5pt} %
		}
\begin{tikzpicture}[bezier bounding box,scale=0.7]
	\begin{pgfonlayer}{nodelayer}
		\node [style=whitenode] (0) at (-6, 3) {};
		\node [style=whitenode] (1) at (-3, 3) {};
		\node [style=whitenode] (2) at (-6, 0) {};
		\node [style=whitenode] (3) at (-3, 0) {};
		\node [style=whitenode] (4) at (-4.5, 4.5) {};
		\node [style=none] (5) at (-4.5, -1) {(a)};
		\node [style=whitenode] (6) at (3, 3) {};
		\node [style=whitenode] (7) at (6, 3) {};
		\node [style=whitenode] (8) at (3, 0) {};
		\node [style=whitenode] (9) at (6, 0) {};
		\node [style=whitenode] (10) at (4.5, 2.5) {};
		\node [style=none] (11) at (4.5, -1) {(b)};
	\end{pgfonlayer}
	\begin{pgfonlayer}{edgelayer}
		\draw [blackedge] (0) to (1);
		\draw [blackedge] (1) to (3);
		\draw [blackedge] (3) to (2);
		\draw [blackedge] (2) to (0);
		\draw [blackedge] (0) to (3);
		\draw [blackedge] (1) to (2);
		\draw [blackedge] (4) to (0);
		\draw [blackedge] (4) to (1);
		\draw [blackedge, bend right=60, looseness=1.25] (4) to (2);
		\draw [blackedge, in=45, out=-15, looseness=1.25] (4) to (3);
		\draw [blackedge] (6) to (7);
		\draw [blackedge] (7) to (9);
		\draw [blackedge] (9) to (8);
		\draw [blackedge] (8) to (6);
		\draw [blackedge] (6) to (9);
		\draw [blackedge] (7) to (8);
		\draw [blackedge] (10) to (6);
		\draw [blackedge] (10) to (7);
		\draw [blackedge, in=105, out=135, looseness=2.75] (8) to (10);
	\end{pgfonlayer}
\end{tikzpicture}



\caption{The only maximal $1$-palne graphs of order $5$}

\label{K4} 
\end{figure}
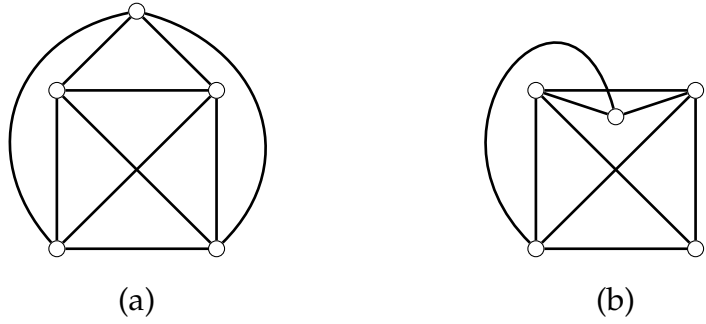 

Now we are going to establish the following claims.
 
\begin{claim}\label{main-cl2} 
	$\hat G$ contains exceptional edges. 
\end{claim} 

\begin{proof}
Suppose that $\hat G$ does not contain exceptional edges.
Assume that $G$ has $h$ hermits.
By Lemma~\ref{h-nest}, 
$\nest(\hat{G})\geq h$.
By Lemma~\ref{ske-max}, 
$\hat G$ is a maximal $1$-plane graph. 
If $n(\hat G)\le 4$, then 
$\nest(\hat{G})=0$, implying that $h=0$
and $n(G)=n(\hat G)=4$, a contradiction
to Claim~\ref{main-cl0}. 
Hence $n(\hat G)\ge 5$.

Since $\hat G$ has no exceptional edges,
each edge in $\hat G$ must 
be contained in some 
$K_4$-subgraph. 
Applying Proposition \ref{skeleton} to $\hat G$ implies that 
$e(\hat G)\ge \frac 73n(\hat G)+\frac 13\nest(\hat G)-3$. 
Since 
$n(G)=n(\hat{G})+h$ and $e(G)=e(\hat{G})+2h$, 
we have 
\equ{e6-4}
{
e(G)
\ge \frac{7}{3}(n(G)-h)+
\frac{1}{3}h-{3}+ 2h 
=\frac{7}{3}n(G)-{3},
}
a contradiction to the assumption 
of $G$.
Claim~\ref{main-cl2} holds. 
\end{proof} 

\begin{claim}\label{main-cl3} 
$G$ is not a counter-example to
	(\ref{main-e1}).
\end{claim} 

\begin{proof}
Assume that $\hat G$ has $k$ exceptional edges. Then 
$k\ge 1$ by Claim~\ref{main-cl2}. 
Let $ab$ be an exceptional  edge in 
$\hat G$.
By Lemma~\ref{skelet-3}, 
$\hat G-ab$ has 
a cut-vertex $f$ and 
two edge-disjoint subgraphs
$H_1$ and $H_2$
such that 
$V(H_1)\cup V(H_2)=V(\hat G)$, 
$V(H_1)\cap V(H_2)=\{f\}$ and 
$E(H_1)\cup E(H_2)=E(\hat G)\setminus \{ab\}$, as shown in Figure~\ref{local-str}.
Lemma~\ref{skelet-3} yields that 
both $H_1$ and $H_2$ are maximal $1$-plane graphs, and 
$\Nest(\hat G)=\Nest(H_1)\cup \Nest(H_2)$.

By Lemma~\ref{2-con1} (b), $\hat G$ is $2$-connected,
implying that 
$\{a,b\}\not\subseteq V(H_i)$ 
for each $i\in [2]$. 
Assume that 
$a\in V(H_1)$ and $b\in V(H_2)$.
Since the boundary of one face in $H_1$ contains vertices 
$a$ and $f$ only,
some edge in $\partial_{H_1}(a)$
crosses  some edge in $\partial_{H_1}(f)$, 
implying that $n(H_1)\ge 4$.
Similarly, $n(H_2)\ge 4$.

Note that $G$ is a graph obtained from
$\hat G$ by putting back all hermits.
By Lemma~\ref{hermit1}, 
each hermit $v$ of $G$ 
appears within a nest $\setn$ 
in $\Nest(\hat G)$.
As $\Nest(\hat G)$ is the disjoint union of $\Nest(H_1)$ and $\Nest(H_2)$,
a hermit $v$ should be put back to $H_1$ if $\setn\in \Nest(H_1)$ and  to $H_2$ otherwise. 
For $i=1,2$, 
let $G_i$ 
denote the maximal $1$-plane graph obtained from $H_i$
by putting all hermits $v$ of $G$ 
that are within  nests of $H_i$.

Assume that $n(G_1)\le n(G_2)$.
We now claim that $n(G_2)\ge 5$.
Suppose that  $n(G_2)=4$.
Then,  $n(H_1)=n(H_2)=4$ and 
$\hat G$ is as shown 
in Figure~\ref{H1-2-4} (a),
where edge $ab$ is non-removable. 
Clearly, the two crossings 
in $\hat G$ cannot produce any 
nest, implying that 
$h=0$ and $G$ is $\hat G$.
But, the graph in Figure~\ref{H1-2-4} (a) 
is not a maximal $1$-plane graph,
as edge $bu$ can be added in to 
get a larger $1$-plane graph $\hat  G+bu$,
as shown in Figure~\ref{H1-2-4} (b),
a contradiction.
Hence $n(G_2)\ge 5$. 

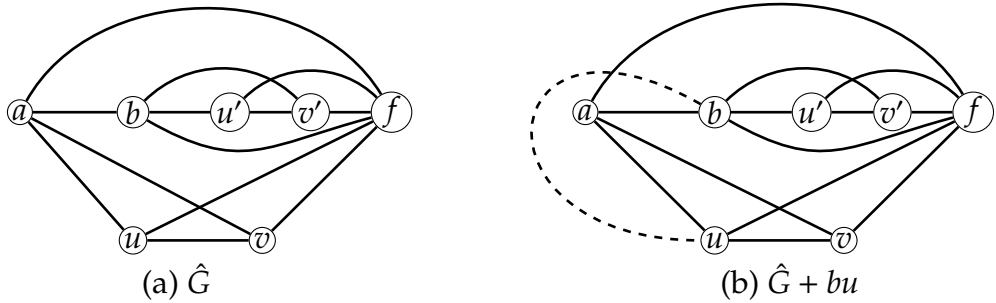
\begin{figure}[!h]
	\centering
	\begin{tikzpicture}[bezier bounding box,scale=0.85]
\tikzset{
			whitenode/.style={circle, draw=black, fill=white, minimum size=0.2em,inner sep=0.5pt},
			bluenode/.style={circle, draw=blue, fill=blue!20, minimum size=5mm},
			blackedge/.style={draw=black, line width=1pt},
		}

	\begin{pgfonlayer}{nodelayer}
		\node [style=whitenode] (0) at (-6.75, 2) {$a$};
		\node [style=whitenode] (1) at (-5, 2) {$b$};
		\node [style=whitenode] (2) at (-3.5, 2) {$u'$};
		\node [style=whitenode] (3) at (-2.25, 2) {$v'$};
		\node [style=whitenode] (4) at (-1, 2) {$f$};
		\node [style=whitenode] (5) at (-5, 0) {$u$};
		\node [style=whitenode] (6) at (-3, 0) {$v$};
		\node [style=whitenode] (7) at (2, 2) {$a$};
		\node [style=whitenode] (8) at (4, 2) {$b$};
		\node [style=whitenode] (9) at (5.5, 2) {$u'$};
		\node [style=whitenode] (10) at (6.75, 2) {$v'$};
		\node [style=whitenode] (11) at (8, 2) {$f$};
		\node [style=whitenode] (12) at (4, 0) {$u$};
		\node [style=whitenode] (13) at (6, 0) {$v$};
	\end{pgfonlayer}
	\begin{pgfonlayer}{edgelayer}
		\draw [blackedge] (0) to (1);
		\draw [blackedge] (1) to (2);
		\draw [blackedge] (2) to (3);
		\draw [blackedge] (3) to (4);
		\draw [blackedge] (0) to (5);
		\draw [blackedge] (0) to (6);
		\draw [blackedge] (4) to (5);
		\draw [blackedge] (4) to (6);
		\draw [blackedge] (6) to (5);
		\draw [blackedge] (7) to (12);
		\draw [blackedge] (7) to (13);
		\draw [blackedge] (13) to (12);
		\draw [blackedge] (7) to (8);
		\draw [blackedge] (8) to (9);
		\draw [blackedge] (9) to (10);
		\draw [blackedge] (10) to (11);
		\draw [blackedge] (13) to (11);
		\draw [blackedge] (11) to (12);
		\draw [blackedge, in=135, out=45] (8) to (10);
		\draw [blackedge, in=135, out=45] (9) to (11);
		\draw [blackedge, in=-165, out=-30, looseness=1.25] (8) to (11);
		\draw [blackedge, in=135, out=45] (2) to (4);
		\draw [blackedge, in=135, out=45] (1) to (3);
		\draw [blackedge, bend left=60] (0) to (4);
		\draw [blackedge, dashed, in=180, out=150, looseness=4.50] (8) to (12);
		\draw [style=blackedge, bend left=60] (7) to (11);
		\draw [style=blackedge, in=-165, out=-30, looseness=1.25] (1) to (4);
	\end{pgfonlayer}
\end{tikzpicture}

	(a) $\hat G$ \hspace{6.5 cm}
	(b) $ \hat G+bu$
	
	\caption{$H_1$ and $H_2$ are the two blocks of $\hat{G}-a b$ with $n\left(H_1\right)=n\left(H_2\right)=4$} 
	
	\label{H1-2-4} 
\end{figure}

By the assumption of $G$,
for each $i\in \brk{1,2}$,
as $4\le n(G_i)<e(G)$, 
$e(G_i)\ge \frac 73 n(G_i)-3$ holds 
whenever $n(G_i)\ge 5$.
As $n(G_2)\ge 5$, this inequality holds 
for $i=2$.
If $n(G_1)=4$, then 
$e(G_1)=\frac 73 n(G_1)-\frac{10}3$,
implying that 
$e(G_1)\ge \frac 73 n(G_1)-\frac{10}3$.
Since $e(G)=e(G_1)+e(G_2)+1$
and $n(G)=n(G_1)+n(G_2)-1$, 
we have 
\equ{main-e2}
{
	e(G)\ge  
\frac 73 n(G_1)-\frac{10}3
+ \frac 73 n(G_2)-3+1
=\frac 73 (n(G)+1)-\frac{16}{3}
=\frac 73 n(G)-3.
}
Hence Claim~\ref{main-cl3} holds.
\end{proof}

Claim~\ref{main-cl3} 
contradicts the assumption of $G$.
It follows that  
(\ref{main-e1}) holds for all maximal $1$-plane graphs 
$G$ with $n(G)\ge 5$. 
Theorem~\ref{main1.3} is proven.
\hfill{$\Box$}

\end{document}